\documentclass[11pt]{amsart}

\usepackage{amsmath, amscd, amssymb}
\usepackage[frame,cmtip,arrow,matrix,line,graph,curve]{xy}
\usepackage{graphpap, color}
\usepackage[mathscr]{eucal}
\usepackage{cancel}
\usepackage{verbatim}

\numberwithin{equation}{section}

\def\sO{{\mathscr O}}

\newcommand{\CC}{\mathbb{C}}

\newcommand{\PP}{\mathbb{P}}
\newcommand{\QQ}{\mathbb{Q}}

\newcommand{\ZZ}{\mathbb{Z}}

\newcommand{\cal}{\mathcal}

\def\cE{{\cal E}}

\def\cH{{\cal H}}

\def\fC{\mathfrak{C}}

\def\loc{\mathrm{loc}}

\def\and{\quad{\rm and}\quad}
\def\lra{\longrightarrow }
\def\mapright#1{\,\smash{\mathop{\lra}\limits^{#1}}\,}
\def\mapleft#1{\,\smash{\mathop{\longleftarrow}\limits^{#1}}\,}

\DeclareMathOperator{\id}{id} 
 
\DeclareMathOperator{\tr}{tr} 

\DeclareMathOperator{\rank}{rank}
\DeclareMathOperator{\spec}{Spec}

\newtheorem{prop}{Proposition}[section]
\newtheorem{theo}[prop]{Theorem}
\newtheorem{lemm}[prop]{Lemma}
\newtheorem{coro}[prop]{Corollary}
\newtheorem{rema}[prop]{Remark}
\newtheorem{exam}[prop]{Example}
\newtheorem{ques}[prop]{Question}
\newtheorem{defi}[prop]{Definition}

\newtheorem{assu}[prop]{Assumption}

\def\beq{\begin{equation}}
\def\eeq{\end{equation}}

\def\dual{^{\vee}}

\def\virt{^{\mathrm{vir}} }
\def\virtloc{\virt_\loc}

\def\bbT{\mathbb{T} }
\def\bbL{\mathbb{L} }
\def\DM{Deligne-Mumford }

\def\bbA{\mathbb{A} }

\def\redd{{\mathrm{red}}}

\def\coker{\mathrm{coker} }
\def\hom{\cH om }

\def\tF{\widetilde{F} }

\def\tf{\widetilde{f} }
\def\tL{\widetilde{L} }

\def\tZ{\widetilde{Z}}

\def\tsig{\tilde{\sigma}}
\def\tY{{\widetilde{Y}}}
\def\DM{Deligne-Mumford }

\def\bk{\mathbf{k}}

\def\ldot{_{\bullet}}
\def\udot{^{\bullet}}
\def\Coh{\mathrm{Coh}\, }
\def\bsig{{\sigma\ldot} }
\def\sp{\mathrm{sp}}
\def\sqe{\sqrt{e}}
\def\sot{e^{\frac12}}
\def\tit{\tilde{t}}
\def\bbE{\mathbb{E}}
\def\bbF{\mathbb{F}}
\def\fe{\mathfrak{e}}

\def\im{\mathrm{im} }

\title[Localizing virtual cycles for DT4]{Localizing virtual cycles for Donaldson-Thomas invariants of Calabi-Yau 4-folds}
\date{}
\author{Young-Hoon Kiem}
\address{Department of Mathematical Sciences and Research Institute
of Mathematics, Seoul National University, Seoul 08826, Korea}
\email{kiem@snu.ac.kr}

\author{Hyeonjun Park}
\address{Department of Mathematical Sciences, Seoul National University, Seoul 08826, Korea}
\email{hyeonjun93@snu.ac.kr}

\thanks{Partially supported by Samsung Science and Technology Foundation SSTF-BA1601-01}

\begin{document}
\begin{abstract} 
Recently Oh-Thomas in \cite{OhTh} constructed a virtual cycle $[X]\virt\in A_*(X)$ for a quasi-projective moduli space $X$ of stable sheaves or complexes over a Calabi-Yau 4-fold against which DT4 invariants \cite{BoJo,CaoLeung} may be defined as integrals of cohomology classes. In this paper, we 
prove that the virtual cycle localizes to the zero locus $X(\sigma)$ of an isotropic cosection $\sigma$ of the obstruction sheaf $Ob_X$ of $X$ and 
construct a localized virtual cycle $[X]\virtloc\in A_*(X(\sigma))$. This is achieved by further localizing the Oh-Thomas class which localizes Edidin-Graham's square root Euler class of a special orthogonal bundle. 
When the cosection $\sigma$ is surjective so that the virtual cycle vanishes, we construct a reduced virtual cycle $[X]\virt_\redd$. 
As an application, we prove DT4 vanishing results for hyperk\"ahler 4-folds. All these results hold for virtual structure sheaves and K-theoretic DT4 invariants. 
\end{abstract}
\maketitle

\section{Introduction}\label{S1}

\subsection{Main results} 
It is a common knowledge that vector bundles and coherent sheaves play a key role in algebraic geometry. 
A fundamental problem about vector bundles and sheaves is to enumerate those with desired properties. For instance, enumeration of ideal sheaves satisfying given certain  conditions constitutes classical enumerative geometry that includes the Schubert calculus. 
Since mid 1990s \cite{LiTi, BeFa, Tho}, counting sheaves on a smooth projective variety $W$ is usually handled by integration of cohomology classes against a special homology class, called the virtual cycle $[X]\virt$, on a moduli space $X$ of stable sheaves over $W$ if $X$ admits such a cycle. When $W$ is a Calabi-Yau 4-fold (CY4 for short), recently Oh-Thomas in \cite{OhTh} constructed a virtual cycle $[X]\virt$ for a projective moduli space $X$ which defines Donaldson-Thomas invariants for CY4 (DT4 invariants for short). At the core of the Oh-Thomas construction lies the localized square root Euler class $\sqe(F,s)$ of a special orthogonal bundle $F$ with an isotropic section $s$ that localizes Edidin-Graham's square root Euler class $\sqe(F)$ constructed in \cite{EdGr}. 

The purpose of this paper is to show that if there is an isotropic cosection of the obstruction theory for $X$, the Oh-Thomas virtual cycle $[X]\virt$ localizes to the zero locus $X(\sigma)$  of the cosection 
\beq\label{i9} \sigma:Ob_X\lra \sO_X\eeq
 of the obstruction sheaf $Ob_X$ of $X$. 
To this end, we first present an alternative construction of the Oh-Thomas class $\sqe(F,s)$ (cf. Definition \ref{n39}), with which 
many properties of $\sqe(F,s)$ are proved in a straightforward manner (cf. Theorem \ref{n51}). Moreover, it is easy to further localize when there is an additional isotropic section $t$ (cf. Theorem \ref{n74}). 
This then leads us to a cosection localized virtual cycle (cf. Theorem \ref{n88}) 
\beq\label{i5} [X]\virtloc\in A_{vd}(X(\sigma))\eeq 
where $vd$ denotes the expected dimension of the virtual cycle. 
If $\sigma$ is surjective so that $X(\sigma)=\emptyset$ and hence $[X]\virt=0$, we construct a reduced virtual cycle (cf. Definition \ref{131}) 
\beq\label{i6} [X]\virt_\redd\in A_{vd+1}(X).\eeq  
Moreover, when the square of the cosection is a nonzero constant, we prove that the virtual cycle always vanishes (cf. Theorem \ref{143}). 
As an application, we obtain vanishing results for DT4 invariants 
when $W$ is a hyperk\"ahler 4-fold (cf. \S\ref{S6}). 
All the results above lift to K-theory 
with appropriate modifications (cf. \S\ref{S8.2} and \S\ref{S8.4}) and virtual classes in intersection theories $\cH_*$ in the sense of \cite{KP} as long as Fulton's conjecture on square root Euler class holds for $\cH_*$ (cf. \S\ref{S8.3}). 


\subsection{Key ideas} For a quasi-projective \DM stack $X$, a symmetric obstruction theory $\phi:\bbE\udot\to \bbL_X$, perfect of amplitude $[-2,0]$, admits a 3-term locally free resolution
$$[B\mapright{d} F\cong F^\vee\mapright{d^\vee} B^\vee]$$
and a cosection $\bsig:\bbE\ldot\to \sO_X[-1]$ is given by a homomorphism $\tsig:F\to \sO_X$ which is a lift of 
$$\sigma=h^1(\bsig):Ob_X=h^1(\bbE\ldot)\lra \sO_X$$
where $\bbE\ldot=(\bbE\udot)^\vee$ is the dual of $\bbE\udot$ (cf. \S\ref{S5}).
By definition, $[F^\vee\mapright{d^\vee} B^\vee]$ is a perfect obstruction theory for $X$ and the intrinsic normal cone $\fC_X$ embeds into the vector bundle stack $[F/B]$ by \cite{BeFa}. 
Let 
\beq\label{i7}C=\fC_X\times_{[F/B]}F\subset F \and Y=C_\redd.\eeq 
Then the virtual cycle of $X$ is defined by Oh-Thomas \cite{OhTh} as 
\beq\label{i4} [X]\virt=\sqe(F|_Y,\tau)[C]\in A_{vd}(X)\eeq 
where $\tau$ is the tautological section and $\sqe(F|_Y,\tau)$ is the localized square root Euler class of the orthogonal bundle $F|_Y$ with an orientation induced from an orientation of $\bbE\udot$ (cf. \cite[\S4]{OhTh}). 
Hence to localize the virtual cycle $[X]\virt$, we have to localize $\sqe(F|_Y,\tau)$ further by the section $\tsig|^\vee_Y$ of $F|_Y$. 

Unfortunately, the construction in \cite[\S3]{OhTh} of localized square root Euler class $\sqe(F,s)$ for a special orthogonal bundle $F$ and an isotropic section $s$ over a \DM stack 
$Y$ is not suitable for a further localization because it already is rather complicated. 
So our first task in this paper is to provide an alternative construction of $\sqe(F,s)$ which allows a further localization by an extra section in a straightforward manner. 
This is achieved by a construction used in \cite{Kis, KLc, KLk, KLq}. 
Letting $\rho:\tY\to Y$ be the blowup along the zero locus $X$ of $s$, we can decompose any $\xi\in A_*(Y)$ into 
\beq\label{i2} \xi=\rho_*\alpha+\imath_*\beta, \quad \alpha\in A_*(\tY), \ \ \beta\in A_*(X)\eeq
where $\imath:X\to Y$ denotes the inclusion. 
For $\beta$, we may apply the ordinary square root Euler class $\sqe(F)$ of Edidin-Graham \cite{EdGr} and obtain $\sqe(F)\beta\in A_*(X)$. On $\tY$, the isotropic section $s$ induces an isotropic subbundle $L=\sO_{\tY}(D)$ of $F|_{\tY}$ where $D$ is the exceptional divisor of $\rho$. Then we have an induced special orthogonal bundle $\tF=L^\perp/L$ over $\tY$ and we may apply $\sqe(\tF)$ to $D\cdot \alpha$. We  let (cf. Definition \ref{n39})
\beq\label{i1}\sqe(F,s)(\rho_*\alpha+\imath_*\beta)=(\rho|_D)_*\Big(\sqe(\tF)D\cdot \alpha\Big)+\sqe(F)\beta \ \ \in\  A_*(X)\eeq
which defines a bivariant class $\sqe(F,s)\in A^n_X(Y)$ ($\rank F=2n$) satisfying $\imath_*\circ \sqe(F,s)=\sqe(F)$ (cf. Theorem \ref{n51}). 
We further prove that $\sqe(F,s)$ constructed in this way coincides with the Oh-Thomas class defined in \cite{OhTh} (cf. Theorem \ref{n55}). 

Now suppose we have two isotropic sections $s,t$ of a special orthogonal bundle $F$ with $s\cdot t=0$. An  obvious way to generalize \eqref{i1} is 
\beq\label{i3}\sqe(F,s;t)(\rho_*\alpha+\imath_*\beta)=(\rho|_{D\cap \tit^{-1}(0)})_*\Big(\sqe(\tF,\tit)D\cdot \alpha\Big)+\sqe(F,t)\beta \eeq
where $\tit$ is the isotropic section of $\tF$ induced by $t$ (cf. Definition \ref{n72}). 
When $s$ and $t$ are independent away from a closed substack $Z\subset Y$ (cf. Definition \ref{n73}), $\sqe(F,s;t)$ is a bivariant class in $A^n_{X\cap Z}(Y)$ (cf. Theorem \ref{n74}). 

When $X$ is a moduli space of stable sheaves or more generally a quasi-projective moduli space of simple perfect complexes on a Calabi-Yau 4-fold $W$ and the standard obstruction theory \eqref{n96} is equipped with an isotropic cosection $\bsig:\bbE\ldot\to \sO_X[-1]$, letting $Y$ denote the reduced stack of the cone $C\subset F$ in \eqref{i7}, the cone reduction lemma in \cite{KLc} tells us that the tautological section $\tau$ of $F|_Y$ and the isotropic section $\tsig|_Y^\vee$ are independent away from $X(\sigma)\times_XY$ and hence we can localize
\eqref{i4} by
\beq\label{i8} [X]\virtloc=\sqe(F|_Y,\tau; \tsig|_Y^\vee)[C]\in A_{vd}(X(\sigma))\eeq
which is well defined and deformation invariant (cf. Theorem \ref{n88}). 
In particular if the cosection \eqref{i9} is surjective so that $X(\sigma)=\emptyset$, we have the vanishing $[X]\virt=0$ (cf. Corollary \ref{126}) and we can define a reduced virtual cycle by 
\beq\label{i10} [X]\virt_\redd=\sqe(L_{\tsig}^\perp/L_{\tsig}|_Y,\tau)[C]\ \ \in\  A_{vd+1}(X)\eeq
where $L_{\tsig}\subset F$ is the trivial line bundle generated by $\tsig^\vee$ (cf. Definition \ref{131}). 
When there is a cosection $\bsig$ which is not isotropic but $\sigma\ldot^2$ is a nonzero constant, the virtual cycle $[X]\virt$ is zero (cf. Theorem \ref{143}). 

There are cosections $\bsig$ associated to holomorphic 2-forms, (3,1)-forms and (0,2)-forms on the Calabi-Yau 4-fold $W$ (cf. \S\ref{S6.1}). We thus obtain localized virtual cycles and vanishing results for DT4 invariants (cf. \S\ref{S6}).  
All the above results hold in K-theory (cf. \S\ref{S8}) and more generally in every intersection theory satisfying Assumption \ref{176}. 


\subsection{The layout} This paper is organized as follows. In \S\ref{Sorth}, we collect necessary facts on orthogonal bundles. In \S\ref{Sn3}, we recall the construction of square root Euler class by Edidin-Graham \cite{EdGr}. 
In \S\ref{Sn4}, we construct a localized square root Euler class by blowup and prove useful properties. In \S\ref{SOT}, we prove that the localized square root Euler class by the blowup construction coincides with the Oh-Thomas class in \cite{OhTh}. In \S\ref{Scos}, we construct square root Euler class localized by two isotropic sections. In \S\ref{S5}, we recall symmetric obstruction theory and 
necessary facts on cosections.  
In \S\ref{S5n}, we construct cosection localized virtual cycle and reduced virtual cycle. We also prove vanishing results. In \S\ref{S6}, cosections for moduli spaces of stable sheaves or perfect complexes on Calabi-Yau 4-folds are constructed and we obtain localization and vanishing results of DT4 invariants. In \S\ref{S8}, we generalize all these results to K-theory and more generally to intersection theories.

%
%
%


\subsection{Convention and notation}\label{S1.4}
All schemes and \DM stacks in this paper are \emph{quasi-projective} over $\CC$. 

Given a Cartesian diagram 
$$\xymatrix{
X\times_YZ\ar[d]\ar[r] & Z\ar[d]\\
X\ar[r]& Y},$$ 
the fiber product $X\times_Y Z$ may be denoted by $Z|_X$ or $X|_Z$ when the meaning is clear from the context. 

We use $\QQ$ as the coefficient ring of Chow groups and K-groups. 
For schemes, everything in this paper works with coefficients in $\ZZ[1/2]$.

We will use the following notation:
\begin{enumerate}
\item  $F$ will denote an orthogonal bundle, 
\item $X$ will often denote a quasi-projective moduli space of stable sheaves or perfect complexes on a Calabi-Yau 4-fold $W$, 
\item $Y$ will often denote an isotropic cone over $X$, contained in an orthogonal bundle over $X$,
\item the tautological section of $F|_Y\to Y$ for $Y\subset F$ will be often denoted by $\tau$.
\end{enumerate}

\bigskip

\noindent\textbf{Acknowledgement}. 
We thank Jeongseok Oh and Richard Thomas for sharing earlier drafts of \cite{OhTh} and illuminating discussions. 
YHK thanks Dennis Borisov, Yalong Cao, Amin Gholampour, Eduardo Gonzalez,  Martijn Kool, Feng Qu, Michail Savvas, Artan Sheshimani, Yukinobu Toda  and Jun Li for useful discussions.

\bigskip
\section{Special orthogonal bundles}\label{Sorth}
In this preliminary section, we collect definitions and facts on orthogonal bundles, orientations and isotropic subbundles that will be used throughout this paper. 

\medskip

Let $F$ be a vector bundle of rank $2n$ on a \DM stack $Y$, 
equipped with a symmetric nondegenerate bilinear pairing
$$q:F\otimes F\lra \sO_Y.$$
For local sections $v$ and $w$ of $F$, we denote the pairing by
$$q(v,w)=q(v\otimes w)=v\cdot w.$$

A subbundle $\Lambda$ of 
$F$ is called \emph{isotropic} if $q(\Lambda\otimes\Lambda)=0$. 
A section $s$ of $F$ is called \emph{isotropic} if $s^2=s\cdot s=0$. 
A closed substack $Z\subset F$ is called \emph{isotropic} if the tautological section 
$\tau$ of $F|_Z=F\times_Y Z$,
induced by the inclusion $Z\to F$ and the identity map $\id_Z$,
is isotropic.  

The pairing $q$ induces an isomorphism $$\hat{q}:F\mapright{\cong} F^\vee$$ of $F$ with its dual $F^\vee$.  The determinant $$\det \hat{q}:\det F\mapright{\cong} \det F^\vee=(\det F)^{-1}$$ 
induces an isomorphism
\beq\label{n1} \sO_Y \mapright{\cong} (\det F)^2.\eeq
An isomorphism 
\beq\label{n2} or:\sO_Y \mapright{\cong} \det F\eeq
whose square is \eqref{n1} is called an \emph{orientation} of $F$.
The triple $(F,q,or)$ is called an \emph{$SO(2n)$-bundle} or a \emph{special orthogonal bundle}.

We will often use the following simple fact.
\begin{lemm}\label{n0}
Let $F$ be a vector bundle over $Y$ with a nondegenerate symmetric bilinear form $q$. Let $s$ be an isotropic section of $F$ whose zero locus $D$ is a Cartier divisor of $Y$ so that we have an inclusion $$\jmath:L=\sO_Y(D)\hookrightarrow F$$ of vector bundles. 
Then the line bundle $L$ is an isotropic subbundle. 
\end{lemm}
\begin{proof}
By the isomorphism $\hat{q}:F\cong F^\vee$, $q(L,L)=0$ is equivalent to the vanishing of the homomorphism 
\[\sO_Y(D)\mapright{\jmath} F\mapright{\hat{q}} F^\vee \mapright{\jmath^\vee} \sO_Y(-D)\]
which becomes
$$s^2=0:\sO_Y\lra \sO_Y(D)\mapright{\jmath} F\mapright{\hat{q}} F^\vee \mapright{\jmath^\vee} \sO_Y(-D)\lra \sO_Y$$
when composed with the inclusions $s:\sO_Y\to \sO_Y(D)$ and $s:\sO_Y(-D)\to\sO_Y$. 
Since $H^0(\sO_Y(-2D))\lra H^0(\sO_Y)$ is injective, we have $q(L,L)=0$.  
\end{proof}

We say that an isotropic subbundle $\Lambda$ of an $SO(2n)$-bundle $F$ is \emph{maximal} if the rank of $\Lambda$ is $n$.    
For a basis $e_1, \cdots, e_n$ for a maximal isotropic subbundle $\Lambda$ over an open substack and its dual basis $f_1,\cdots, f_n$ for the dual $\Lambda^\vee$ satisfying $e_i\cdot f_j=\delta_{ij}$, 
we have 
\beq\label{n7} e_1\wedge f_1\wedge \cdots \wedge e_n\wedge f_n=\pm or.\eeq
We say that a maximal isotropic subbundle $\Lambda$ is \emph{positive} (resp. \emph{negative}) if the sign is positive (resp. negative). 

\medskip

Let $K$ be an isotropic subbundle of $F$ and let $K^\perp$ denote 
the orthogonal complement of $K$ in $F$ with respect to $q$. 
Then we have exact sequences 
\beq\label{n3} 0\lra K^\perp \lra F\lra K^\vee\lra 0,\eeq
\beq\label{n4} 0\lra K\lra K^\perp \lra K^\perp/K\lra 0.\eeq
The bilinear form $q$ on $F$ induces a nondegenerate bilinear form 
\beq\label{n5} \bar{q}:K^\perp/K\otimes K^\perp/K \lra \sO_Y\eeq
and the orientation \eqref{n2} induces an orientation
\beq\label{n6} \bar{or}: \sO_Y\mapright{\cong} \det (K^\perp/K)\eeq
for $K^\perp/K$ by the isomorphisms
$$\det F\cong \det K^\perp\otimes \det K^\vee\cong \det (K^\perp/K)\otimes \det K \otimes (\det K)^{-1} \cong \det (K^\perp/K)$$
from \eqref{n3} and \eqref{n4}. 
In case $K$ is maximal isotropic so that $K^\perp/K=0$ and $\det (K^\perp/K)\cong \sO_Y$, the induced orientation 
$$\bar{or}:\sO_Y\mapright{\cong} \det(K^\perp/K)=\sO_Y$$
is $1$ if $K$ is positive and $-1$ if $K$ is negative.

\bigskip

\section{Square root Euler class}\label{Sn3}

In this section, we recall the construction of the Gysin map $0^!_V$, the Euler class $e(V)$ and the localized Euler class $e(V,s)$ for a vector bundle $V$ and a section $s$ of $V$. 
We further recall the Edidin-Graham class which is a square root of the Euler class of an $SO(2n)$-bundle.

\subsection{Bivariant classes}\label{nS3.1}

We recall basics on bivariant classes from \cite[Chapter 17]{Ful} and \cite[\S5]{Vist}. 
For \DM stacks $X$ and $Y$, $A_*(X)$ and $A_*(Y)$ denote their Chow groups. 
\begin{defi} 
A \emph{bivariant class} $c\in A^p(X\mapright{f} Y)$ for a morphism $f:X\to Y$ is a collection of homomorphisms
$$c_g:A_*(Y')\lra A_{*-p}(X')$$
for any Cartesian square 
$$\xymatrix{
X'\ar[r]^{f'}\ar[d]_{g'} & Y'\ar[d]^{g}\\
X\ar[r]^f & Y
}$$
compatible with proper pushforward, flat pullback and refined intersection. 
We let 
$A^p(X)=A^p(X\mapright{\id_X} X)$. 
When $f$ is a closed immersion $\imath:X\to Y$, we let  
$$A^p_X(Y)=A^p(X\mapright{\imath} Y).$$
\end{defi}
The Chern classes $c_p(V)\in A^p(Y)$ of vector bundles $V$ on $Y$ are bivariant classes,
that commute with any bivariant classes and Gysin maps.

\medskip
\subsection{Localized Euler class and Gysin map}\label{Sn3.2}

Let $Y$ be a \DM stack and $\pi:V\to Y$ be a vector bundle of rank $r$ on $Y$. 
The flat pullback $\pi^*$ is an isomorphism by \cite[Theorem 3.3]{Ful} or \cite[Corollary 2.5.7]{Kresch} whose inverse is denoted by 
\beq\label{n11}
0_V^!=(\pi^*)^{-1}:A_*(V)\lra A_{*-r}(Y).
\eeq
By \cite[Example 3.3.2]{Ful}, the Euler class of $V$ is the composite
\beq\label{n12}
e(V)=c_r(V):A_*(Y)\mapright{(0_V)_*} A_*(V)\mapright{0^!_V} A_{*-r}(Y)
\eeq where $0_V:Y\to V$ denotes the inclusion of the zero section. 

\medskip

The Euler class $e(V)=c_r(V)\in A^r(Y)$ localizes to the zero locus of a section $s\in H^0(V)$. Let 
\beq\label{n13}
\imath:X=s^{-1}(0)\lra Y\eeq
denote the zero locus of $s$, i.e. the closed substack of $Y$ defined by the image of $s^\vee:V^\vee\to \sO_Y$. Let 
\beq\label{n53} \sp:A_*(Y)\lra A_*(C_{X/Y})\eeq
be the specialization homomorphism \cite[Proposition 5.2]{Ful}, where $C_{X/Y}$ denotes the normal cone of $X$ in $Y$. 
As $X=s^{-1}(0)$, $C_{X/Y}\subset V|_X$ and we have the localized Euler class \cite[Chapter 6]{Ful}
\beq\label{n15}
e(V,s):A_*(Y)\mapright{\sp} A_*(C_{X/Y})\lra A_*(V|_X)\mapright{0^!_{V|_X}} A_{*-r}(X).\eeq
Moreover, for any morphism $g:Y'\to Y$ and $X'=X\times_YY'$, we have homomorphisms
$$e(V,s):A_*(Y')\mapright{\sp} A_*(C_{X'/Y'})\lra A_*(C_{X/Y}|_{X'})\lra A_*(V|_{X'})\mapright{0^!_{V|_{X'}}} A_{*-r}(X')$$
using the natural embedding $C_{X'/Y'}\subset C_{X/Y}|_{X'}$. 
By \cite[Chapter 6]{Ful}, these homomorphisms define a bivariant class, called the \emph{localized Euler class} 
\beq\label{n14}
e(V,s)\in A^r(X\mapright{\imath} Y)= A^r_X(Y)\eeq 
whose pushforward by $\imath$ is the Euler class $e(V)=c_r(V)$. 
It is obvious from the definitions that the localized Euler class equals the refined Gysin map 
\beq\label{n33} 0_V^!=e(V,s):A_*(Y)\lra A_{*-r}(X)\eeq
by the Cartersian square 
$$\xymatrix{
X\ar[r]\ar[d] & Y\ar[d]^s\\
Y\ar[r]^{0_V} & V.}$$

\medskip

By \eqref{n12}, the Euler class $e(V)$ can be obtained from the Gysin map $0^!_V$.    
Conversely, the Gysin map $0^!_V$ is obtained from the localized Euler class by 
$$e(V|_V,\tau):A_*(V)\lra A_{*-r}(Y)$$
where $V|_V=\pi^*V=V\times_YV$ and $\tau$ is the tautological section of $V|_V\to V$
whose zero locus is $Y$. The identity
\beq\label{n14.1}
e(V|_V,\tau)=0_V^!\eeq
follows from the definition \eqref{n15} as $C_{Y/V}=V$ and $\sp=\id_{A_*(V)}$. 

By construction, for any section $s\in H^0(V)$ and $X=s^{-1}(0)$, we have the equality 
\beq\label{n47} e(V)=e(V,s)\circ \imath_*:A_*(X)\mapright{\imath_*}A_*(Y)\mapright{e(V,s)} A_{*-r}(X)\eeq
by \eqref{n12} since the composition $$A_*(X)\mapright{\imath_*} A_*(Y)\mapright{\sp} A_*(C_{X/Y})\lra A_*(V|_X)$$ equals $(0_V)_*$.  

\medskip

\subsection{The square root Euler class of Edidin-Graham}\label{S2.1} 

In this subsection, we recall the construction of square root Euler class by Edidin-Graham in \cite{EdGr}.

Let $Y$ be a \DM stack and $(F,q,or)$ be an $SO(2n)$-bundle on $Y$. 
Following \cite{EdGr}, let $Q_{n-1}\subset \PP F$ denote the quadric bundle defined by the vanishing $q(v,v)=0$. Let $$F_1=\sO_{\PP F}(-1)|_{Q_{n-1}}\subset F|_{Q_{n-1}}$$ 
be the tautological line bundle and let $F_1^\perp$ be its orthogonal complement with respect to $q$. 
Then $F_1\subset F_1^\perp$ and $F_1^\perp/F_1$ is an $SO(2n-2)$-bundle by \eqref{n5} and \eqref{n6}. 
Let $$Q_{n-2}\subset \PP (F_1^\perp/F_1)$$
denote the quadric bundle defined by the vanishing of the quadratic form on $F_1^\perp/F_1$. 
Let $$F_2=\sO_{\PP (F_1^\perp/F_1)}(-1)|_{Q_{n-2}}\times_{F_1^\perp/F_1} F_1^\perp \subset F|_{Q_{n-2}}$$ 
and let $F_2^\perp$ be the orthogonal complement of $F_2$ in $F|_{Q_{n-2}}$. 
Then $F_2\subset F_2^\perp$ and $F_2^\perp/F_2$ is an $SO(2n-4)$-bundle.  
Continuing this way, we obtain a tower
\beq\label{55}\xymatrix{
Q_1\ar[r]\ar[dr]_{p_1} & \PP(F_{n-2}^\perp/F_{n-2})\ar[d]\\ 
&Q_2\ar[r]& \cdots \\
&&\cdots & \ar[d] \\
&& & Q_{n-2} \ar[r]\ar[dr]_{p_{n-2}} & \PP (F_1^\perp/F_1)\ar[d]\\
&&&& Q_{n-1}\ar[r]\ar[dr]_{p_{n-1}} & \PP F\ar[d]\\
&&&&& Y
}\eeq
whose horizontal arrows are inclusions and vertical arrows are the bundle projection maps. 

Let $Q=Q_1$ and $p=p_{n-1}\circ p_{n-2}\circ\cdots \circ p_1$.
By abuse of notation, we denote by $F_i$ the pullback of the rank $i$ isotropic subbundle $F_i$ to $Q$ so that we have a flag
$$F_1\subset F_2\subset \cdots \subset F_n \subset F|_Q$$
of isotropic subbundles with $\rank F_i=i$ where $F_n$ is the unique positive maximal isotropic subbundle of $F|_Q$ containing $F_{n-1}$. Let $\Lambda=F_n$ so that we have an exact sequence
\beq\label{59}
0\lra \Lambda \lra F|_Q\lra \Lambda^\vee\lra 0.\eeq

Let $h_i=c_1(\sO_{Q_i}(1))\in A^1(Q_i)$ for $1\le i\le n-1$. Since $Q_i$ are quadric hypersurface bundles, we have
\beq\label{56}
(p_i)_*(h_i^{2i}/2\cap p_i^*(\xi))=\xi,\quad \forall \xi\in A_*(Q_{i+1}).
\eeq
where $Q_n=Y$. We thus have the identity 
\beq\label{57}
p_*(h\cap p^*(\xi))=\xi ,\quad \forall \xi\in A_*(Y). 
\eeq
where 
\beq\label{n19} h=h_1^2h_2^4\cdots h_{n-1}^{2n-2}/2^{n-1}.\eeq
Edidin and Graham in \cite{EdGr} defined the square root Euler class of $F$ as follows. 
\begin{defi} \cite{EdGr} For an $SO(2n)$-bundle $F$ on $Y$, we have a bivariant class
\beq\label{60}
\sqrt{e}(F):A_*(Y)\lra A_{*-n}(Y),\eeq
 in $A^n(Y)$ defined by 
$$\sqrt{e}(F)(\xi)=p_*\left(e(\Lambda) h\cap p^*(\xi)\right)
=p_*\left((-1)^n0^!_{\Lambda^\vee}(0_{\Lambda^\vee})_* (h\cap p^*(\xi))\right)$$
 which satisfies  
$$\sqrt{e}(F)^2=(-1)^ne(F).$$ 
\end{defi}
By \cite[Theorem 1 (c)]{EdGr}, if $F$ admits a positive maximal isotropic subbundle $V$ so that we have
an exact sequence 
\beq\label{63}
0\lra V\lra F\lra V^\vee\lra 0,\eeq
we have the equality
\beq\label{61}
\sqrt{e}(F)=e(V)=(-1)^n0_{V^\vee}^!(0_{V^\vee})_*\in A^n(Y).
\eeq
More generally, if $K$ is an isotropic subbundle of $F$, we have  
\beq\label{n42} \sqe(F)=e(K)\sqe(K^\perp/K)=\sqe(K^\perp/K) e(K).\eeq
Indeed, by lifting the identity to the maximal isotropic flag variety of $K^\perp/K$ (cf. \eqref{55}), we may assume that $F$ admits a positive maximal isotropic subbundle $V$ containing $K$ with $V/K$ maximal isotropic in $K^\perp/K$, so that 
$e(V)=e(K)e(V/K)$ which implies \eqref{n42}. 

\begin{rema}
In \cite{EdGr}, the square root Euler class $\sqe(F)$ was defined for schemes but everything in \cite{EdGr} works over quasi-projective \DM stacks \cite{KrDM} as follows. Recall that we consider quasi-projective \DM stacks only in this paper as we declared in \S\ref{S1.4}. 

Let $F$ be an $SO(2n)$-bundle over a quasi-projective \DM stack $Y$. Since $Y$ is a quotient stack, there exists a vector bundle $V$ of rank $r$ over $Y$ and an open subscheme $U \subseteq V$ such that the codimension of $V\setminus U$ in $V$ is large enough. 
Hence for fixed $i$, $A_i(Y) = A_{i+r}(V)=A_{i+r}(U)$ and we can define the square root Euler class $\sqrt{e}(F) : A_* (Y) \to A_{*-n} (Y)$ by $\sqrt{e}(F|_U)$. 
This construction is independent of the choice of $V$ and $U$. Indeed, if $U$ and $U'$ are two such choices, then considering the fiber product $U\times_Y U'$ proves the claim because $\sqrt{e}(F|_U)$ commutes with smooth pullbacks. 

With this definition, it is easy to show that the square root Euler class $\sqe(F)$ over a \DM stack $Y$ commutes with proper pushforwards, flat pullbacks, and lci pullbacks, and \eqref{n42} holds for any isotropic subbundle $K$ of $F$. 
\end{rema}

%

\bigskip

\section{Localized square root Euler class}\label{Sn4}

In this section, we construct a localized square root Euler class
$$\sqrt{e}(F,s)\in A^n_X(Y)$$ for an $SO(2n)$-bundle $F$ and an isotropic section $s\in H^0(F)$ with $X=s^{-1}(0)$ by using a method in \cite{Kis, KLc, KLk, KLq}. 
We will prove in \S\ref{SOT} that the square root Euler class $\sqrt{e}(F,s)$ defined in this section coincides with the Oh-Thomas class constructed in \cite{OhTh}. 

\medskip

\subsection{Localized Euler class by blowup}\label{Sn4.1}

To warm up, let us consider an alternative construction of the localized Euler class \eqref{n15} by a blowup. Let $V$ be a vector bundle of rank $r$ on $Y$ and $s$ be a section of $V$ with $X=s^{-1}(0)$. 

Let $\rho:\tY\to Y$ be the blowup of $Y$ along $X$ and let $D$ denote the exceptional divisor so that we have a Cartesian diagram
\beq\label{n30}\xymatrix{
D\ar[r]^{\jmath} \ar[d]_{\rho'} &\tY\ar[d]^\rho\\
X\ar[r]^{\imath}\ar[d]_\imath & Y\ar[d]^{s}\\
Y\ar[r]^{0_V} & V}.\eeq
By \cite[Proposition 2.3.6]{Kresch} (cf. \cite[Example 1.8.1]{Ful}), we have an exact sequence
\beq\label{n31}
A_*(D)\lra A_*(\tY)\oplus A_*(X)\lra A_*(Y)\lra 0
\eeq
whose arrows are $$\gamma\mapsto (\jmath_*\gamma, -\rho'_*\gamma) \and (\alpha,\beta)\mapsto \rho_*\alpha+\imath_*\beta.$$ 

The dual  $s^\vee:V^\vee\twoheadrightarrow I_X\subset \sO_Y$ pulls back to a surjective homomorphism 
$$V|_{\tY}^\vee\lra \sO_{\tY}(-D)$$
of locally free sheaves. Its dual fits into an exact sequence
\beq\label{n32} 0\lra \sO_{\tY}(D)\lra V|_{\tY}\lra \bar V\lra 0\eeq 
where $\bar V=V|_{\tY}/\sO_{\tY}(D)$. 

Using \eqref{n31}, we define a homomorphism $\mathbf{e}(V,s):A_*(Y)\lra A_{*-r}(X)$ by 
\beq\label{n36}\begin{aligned}\mathbf{e}(V,s)(\xi)&=\rho'_*e(\bar V) \jmath^*\alpha+e(V) \beta\\ 
&=\rho'_*\jmath^*e(\bar V) \alpha+e(V) \beta\end{aligned}\eeq 
for  $\xi=\rho_*\alpha+\imath_*\beta$ with $\alpha\in A_*(\tY)$, $\beta\in A_*(X)$. Here $\jmath^*\alpha=D\cdot\alpha$. 

From \eqref{n30} and \eqref{n32}, we have
$$0_V^!=e(\bar V) \jmath^*:A_*(\tY)\lra A_{*-r}(D)$$
by the excess intersection formula \cite[Theorem 6.3]{Ful}. 
For $\alpha\in A_*(\tY)$, 
\beq\label{n34}\rho'_*(e(\bar V) \jmath^*\alpha) = \rho'_*(0_V^! \alpha)=0_V^!\rho_*\alpha=e(V,s)\cap \rho_*\alpha.\eeq
For $\beta\in A_*(X)$, we have the equality
\beq\label{n35}
e(V) \beta=0_V^!\beta=0_V^!\imath_*\beta= e(V,s)\cap \imath_*\beta\eeq
from the Cartesian diagam 
$$\xymatrix{
X\ar[r]^\id \ar[d]_\id& X\ar[d]^{\imath}\\
X\ar[r]^\imath\ar[d]_\imath & Y\ar[d]^s\\
Y\ar[r]^{0_V} & V.
}$$
By \eqref{n34} and \eqref{n35}, we find that $\mathbf{e}(V,s)$ in \eqref{n36} equals the localized Euler  class $e(V,s)$ in \eqref{n15}. In particular, \eqref{n36} is independent of the choice of $\alpha$ and $\beta$ satisfying $\xi=\rho_*\alpha+\imath_*\beta$, because $e(V,s)(\xi)$ is.

\medskip

\subsection{Cosection localized Gysin map}\label{Sn4.2}

A similar blowup construction was used in \cite{KLc, KLk, KLq} for cosection localized Gysin maps, which we recall now. 

Let $V$ be a vector bundle of rank $r$ over $Y$ and $\sigma:V\to \sO_Y$ be a cosection. Let $Y(\sigma)$ be the zero locus of $\sigma^\vee\in H^0(V^\vee)$ and 
$$V(\sigma)=V|_{Y(\sigma)}\cup \ker(\sigma|_{Y-Y(\sigma)})\subset V.$$
Let $\tY$ be the blowup of $Y$ along $Y(\sigma)$ so that we have a Cartesian square 
\beq\label{n37}\xymatrix{
D\ar[r]^{\jmath} \ar[d]_{\rho'} &\tY\ar[d]^\rho\\
Y(\sigma)\ar[r]^{\imath} & Y}\eeq
and a surjection $\sigma|_{\tY}:V|_{\tY}\to \sO_{\tY}(-D)$ whose kernel is denoted by 
$V'$. Let $\tilde{\rho}:V'\to V(\sigma)$ denote the restriction of the natural morphism $V|_{\tY}\to V$ induced by $\rho$. Let $\tilde{\imath}:V|_{Y(\sigma)}\to V(\sigma)$ denote the inclusion map.

The cosection localized Gysin map is a homomorphism
$$0^!_{V,\sigma}:A_*(V(\sigma))\lra A_{*-r}(Y(\sigma))$$
defined by 
\beq\label{n38} 0^!_{V,\sigma}(\xi)=-\rho'_*(\jmath^*0^!_{V'}\alpha)+0_{V|_{Y(\sigma)}}^! \beta
=-\rho'_*(0_{V'}\circ \jmath)^!\alpha + 0_{V|_{Y(\sigma)}}^!\beta\eeq
where 
$$\xi=\tilde{\rho}_*\alpha+\tilde{\imath}_*\beta 
\quad \text{for } \alpha\in A_*(V'), \beta\in A_*(V|_{Y(\sigma)}).$$ 
It was shown in \cite{KLc, KLk, KLq} that $0^!_{V,\sigma}(\xi)$ is independent of the choices of $\alpha$ and $\beta$. 

It is straightforward to check that $0^!_{V,\sigma}$ commutes with proper pushforward and flat pullback. Indeed, if $g:Y_1\to Y$ is a proper morphism, letting $V_1$ and $\sigma_1$ denote the pullbacks of $V$ and $\sigma$ by $g$, we have a commutative diagram
$$\xymatrix{
A_*(V_1(\sigma_1))\ar[r]^{0^!_{V_1,\sigma_1}} \ar[d]_{(g_1)_*} & A_{*-r}(Y_1(\sigma_1))\ar[d]^{(g|_{Y_1(\sigma_1)})_*}\\
A_*(V(\sigma))\ar[r]^{0^!_{V,\sigma}} &A_{*-r}(Y(\sigma)).
}$$ where $g_1:V_1(\sigma_1)\to V(\sigma)$ is the induced morphism. 
If $g:Y_1\to Y$ is flat, we have the equality 
$$0^!_{V_1,\sigma_1}\circ g_1^*=(g|_{Y_1(\sigma_1)})^*\circ 0^!_{V,\sigma}.$$

From \eqref{n36} and \eqref{n38}, it is straightforward that if $V$ is a vector bundle of rank $r$ and $s$ is a section of $V$, then we have the equality 
\beq\label{n52}
e(V,s)=(-1)^r0^!_{V^\vee,s^\vee}(0_{V^\vee})_*.\eeq

\medskip

\subsection{Localized square root Euler class by blowup}\label{Sn4.3}

By the same method as we saw above in \S\ref{Sn4.1} and \S\ref{Sn4.2}, we now construct a square root Euler class of an $SO(2n)$-bundle localized by an isotropic section. 

Let $(F,q,or)$ be an $SO(2n)$-bundle over $Y$ and let $s$ be an isotropic section of $F$. 
Let $X=s^{-1}(0)$ be the zero locus of $s$ and $\rho:\tY\to Y$ be the blowup of $Y$ along $X$ so that we have a Cartesian square
\[\xymatrix{
D\ar[r]^\jmath \ar[d]_{\hat{\rho}} & \tY\ar[d]^\rho\\
X\ar[r]^\imath & Y.}\]
The dual $s^\vee$ of $s$ induces a surjection $F^\vee|_{\tY}\to \sO_{\tY}(-D)$ whose dual 
$$L:= \sO_{\tY}(D)\hookrightarrow F|_{\tY}$$ is the inclusion of an isotropic line subbundle by Lemma \ref{n0}. Let 
$$\tF=L^\perp/L$$ be the induced $SO(2n-2)$-bundle by \S\ref{Sorth}. 

Using \eqref{n31}, the localized square root Euler class $\sqe(F,s)$ can be defined by \eqref{n36} with $e$ replaced by $\sqe$ and $\bar V$ by the orthogonal bundle $\tF$. 
\begin{defi}\label{n39}
The \emph{localized square root Euler class} of the $SO(2n)$-bundle $(F,q,or)$ is defined by
\beq\label{n40} 
\sqrt{e}(F,s):A_*(Y)\lra A_{*-n}(X),\eeq
\beq\label{n41}
\sqrt{e}(F,s)(\xi)={\hat{\rho}}_*\jmath^*\sqe(\tF)\alpha+\sqe(F) \beta
\eeq
for $\xi=\rho_*\alpha+\imath_*\beta$ with 
$\alpha\in A_*(\tY)$, $\beta\in A_*(X)$ by \eqref{n31}.
\end{defi}

The following lemma shows that \eqref{n41} is well defined.
\begin{lemm}\label{n43}
The localized square root Euler class $\sqe(F,s)(\xi)$ is independent of the choices of $\alpha$ and $\beta$ satisfying $\xi=\rho_*\alpha+\imath_*\beta$.
\end{lemm}
\begin{proof}
Let $\rho_*\alpha+\imath_*\beta=\rho_*\alpha'+\imath_*\beta'$
so that
$\rho_*(\alpha-\alpha')=\imath_*(\beta'-\beta)$.
By \eqref{n31}, there exists $\gamma\in A_*(D)$ such that $\jmath_*\gamma=\alpha-\alpha'$ and ${\hat{\rho}}_*\gamma=\beta'-\beta$. 

Since $L=\sO_{\tY}(D)$ is isotropic in $F|_{\tY}$, by \eqref{n42}, we have
$$\sqe(F|_{\tY})=\sqe(\tF)e(L)$$
and hence 
$$\begin{aligned}
{\hat{\rho}}_*\jmath^*\sqe(\tF)(\alpha-\alpha')&={\hat{\rho}}_*\jmath^*\sqe(\tF)\jmath_*\gamma
={\hat{\rho}}_*\sqe(\tF)\jmath^*\jmath_*\gamma\\
&={\hat{\rho}}_*\sqe(\tF)e(L)\gamma ={\hat{\rho}}_*\sqe(F|_{\tY})\gamma\\
&=\sqe(F) {\hat{\rho}}_*\gamma=\sqe(F) (\beta'-\beta)\end{aligned}$$
which proves the lemma.
\end{proof}

The following lemma shows that $\sqe(F,s)$ is indeed a localization of the Edidin-Graham class $\sqe(F)$.
\begin{lemm}\label{n44}
$\imath_*\circ \sqe(F,s)=\sqe(F):A_*(Y)\to A_{*-n}(Y)$.
\end{lemm}
\begin{proof}
By $\sqrt{e}(F)=e(L)\sqrt{e}(\tF)$ (cf. \eqref{n42}) and the identity $$\jmath_*\jmath^*=c_1(L)=e(L),$$ using the notation above, we have 
\begin{align*}
\imath_*\sqrt{e}(F,s)(\rho_*\alpha + \imath_* \beta)  
& = \imath_*{\hat{\rho}}_* \jmath^* \sqrt{e}(\tF) \alpha + \imath_*\sqrt{e}(F)\beta\\  
& = \rho_* \jmath_*\jmath^*\sqrt{e}(\tF) \alpha + \sqrt{e}(F)\imath_*\beta\\
& = \rho_* e(L)\sqrt{e}(\tF) \alpha + \sqrt{e}(F)\imath_*\beta\\
& = \sqrt{e}(F) (\rho_*  \alpha + \imath_*\beta).
\end{align*}
which completes the proof.
\end{proof}

Next we prove that $\sqe(F,s)$ is a bivariant class in $A^n_X(Y)$ (cf. \S\ref{nS3.1}).
\begin{lemm}\label{102}
Let $F$ be an $SO(2n)$-bundle on $Y$ with an isotropic section $s$ whose zero locus is $X$. 
Consider a Cartesian diagram of \DM stacks
$$\xymatrix{X' \ar[d]_{h} \ar[r]^{\imath'} & Y' \ar[d]^f \ar[r] & Z'\ar[d]^g \\
X \ar[r]^{\imath} & Y \ar[r] & Z.}$$
\begin{enumerate}
\item If $f:Y'\to Y$ is a proper morphism, then 
\begin{equation}\label{102.1}
\sqrt{e}(F,s) \circ f_* = h_* \circ \sqrt{e}(f^*F,f^*s).    
\end{equation}
\item If $f:Y'\to Y$ is an equi-dimensional flat morphism, then 
\begin{equation}\label{102.2}
h^* \circ \sqrt{e}(F,s) =  \sqrt{e}(f^*F,f^*s) \circ f^*.    
\end{equation}
\item If $g:Z'\to Z$ is a local complete intersection morphism, then 
\begin{equation}\label{102.3}
g^! \circ \sqrt{e}(F,s) =  \sqrt{e}(f^*F,f^*s)\circ g^!.
\end{equation}
\end{enumerate}
\end{lemm}


\begin{proof}
We keep using the notation of Definition \ref{n39}.
Let $\rho:\tY\to Y$ and $\rho':\tY'\to Y'$ be the blowups of $Y$ and $Y'$ along $X=s^{-1}(0)$ and $X'=(f^*s)^{-1}(0)$ respectively. Let $D'\subset \tY'$ and $D\subset \tY$ be the exceptional divisors. Then there is a morphism $\tf : \tY' \to \tY$ satisfying
$f\circ\rho'=\rho\circ \tf$.  
Consider the Cartesian diagrams
$$\xymatrix{
D' \ar[r]^{\jmath'} \ar[d]_{\hat\rho'} & \tY' \ar[d]^{\rho'} && D' \ar[r]^{\jmath'} \ar[d]_{\widetilde{h}} & \tY' \ar[d]^{\tf}\\
X' \ar[r]^{\imath'} & Y' && D \ar[r]^{\jmath} & \tY.
}$$
Then we have $\sO_{\tY'}(D')=\tf^*L$ and ${\jmath'}^!=\jmath^!:A_*(\tY') \to A_{*-1}(D').$


(1) If $f:Y' \to Y$ is proper, then for any $\alpha \in A_* (\tY')$, we have
$$\begin{aligned}
\hat\rho_{*}  \jmath^* \sqrt{e}(\tF)  \tf_*   \alpha &=\hat\rho_{*}  \jmath^* \tf_* \sqrt{e}(\tF)    \alpha \\
&=\hat\rho_{*}  \tilde{h}_*{\jmath'}^* \sqrt{e}(\tF)    \alpha \\
&= h_* \hat{\rho}'_{*} {\jmath'}^* \sqrt{e}(\tF)  \alpha\end{aligned}$$
which implies that \eqref{102.1} holds for $\rho'_{*}\alpha$. 
For $\beta \in A_* (X')$, \eqref{102.1} holds for $\imath'_{*}\beta$ 
because $\sqrt{e}(F)$ is bivariant.

(2) If $f:Y' \to Y$ is flat, then $\tY'=\tY\times_Y Y'$ by \cite[Lemma 69.17.3]{StPr}. Since every operation in \eqref{n41} commutes with flat pullbacks, we have \eqref{102.2}.

(3) Now assume that $g:Z' \to Z$ is an lci morphism. Consider the Cartesian diagram 
$$\xymatrix{
X'' \ar[r] \ar[d]_r &Y'' \ar[r]^{f'} \ar[d]^q & \tY \ar[d]^{\rho} \\
X' \ar[r]^{\imath} & Y' \ar[r]^f  & Y.}$$ 
For any $\alpha \in A_* (\tY)$, we have
\begin{align*}
\sqrt{e}(f^*F,f^*s)g^!\rho_*\alpha = \sqrt{e}(f^*F,f^*s)q_*g^!\alpha = r_*\sqrt{e}(F|_{Y''},s|_{Y''})g^!\alpha
\end{align*}
by \eqref{102.1} above. 
As the induced section $s|_{Y''}$ is contained in $L|_{Y''}$, \eqref{103.1} in Lemma \ref{103} below gives us
\begin{align*}
r_* \sqrt{e}(F|_{Y''},s|_{Y''})g^!\alpha 
= r_*\sqrt{e}(\tF)\jmath^! g^!\alpha
= g^! \hat{\rho}_*\sqrt{e}(\tF) \jmath^! \alpha,
\end{align*}
which proves \eqref{102.3} for $\rho_*\alpha$. When $\beta \in A_* (X)$,  \eqref{102.3} holds for $\imath_*\beta$ since $\sqrt{e}(F)$ is bivariant. 
\end{proof}


\begin{lemm}\label{103}
Let $K$ be an isotropic subbundle of the $SO(2n)$-bundle $F$ over $Y$ such that $s \cdot K=0$, i.e. $s \in H^0(K^{\perp}).$ Let $s_1 \in H^0(K^{\perp}/K)$ be the induced isotropic section and 
let $Z=s_1^{-1}(0)$ and $X=s^{-1}(0)$ denote the zero loci of $s_1$ and $s$ respectively. 
Let $s_2=s|_{Z} \in H^0(Z,K|_{Z})$ be the restriction. Then we have
\begin{equation}\label{103.2}
\sqrt{e}(F,s) = e(K|_{Z},s_2) \circ \sqrt{e}(K^{\perp}/K,s_1):A_*(Y)\lra A_{*-n}(X)
\end{equation}

In particular, if $s \in H^0(K)$, then we have
\begin{equation}\label{103.1}
\sqrt{e}(F,s)=\sqrt{e}(K^{\perp}/K)\circ e(K,s).    
\end{equation}
\end{lemm}

\begin{proof}
We keep using the notation of Definition \ref{n39}. 
Let us first prove \eqref{103.1} and complete the proof of Lemma \ref{102}. 
Suppose $s \in H^0(K)$. Then $L=\sO_{\tY}(D)$ is a subbundle of $\rho^*K$. Hence the quotient bundle $\rho^*K/L$ is an isotropic subbundle of the reduction $\tF=L^{\perp}/L$ and 
$$(\rho^*K/L)^{\perp}/(\rho^*K/L)\cong \rho^*(K^{\perp}/K).$$ 
For $\alpha \in A_*(\tY)$, by \eqref{n42} and the excess intersection formula, 
we have 
\beq\label{n45}\begin{aligned}
\hat{\rho}_* \sqrt{e}(\tF) \jmath^* \alpha & = \hat{\rho}_* \sqrt{e}(K^{\perp}/K) e(\rho^*K/L) \jmath^*\alpha \\
&=\hat{\rho}_* \sqrt{e}(K^{\perp}/K) e(K,s)\alpha\\
&=\sqrt{e}(K^{\perp}/K)e(K,s)\rho_*\alpha.\end{aligned}\eeq
For $\beta \in A_* (X)$, we have
\beq\label{n46}\sqrt{e}(F)\beta = \sqrt{e}(K^{\perp}/K)e(K)\beta=\sqrt{e}(K^{\perp}/K)e(K,s)\imath_*\beta\eeq
by \eqref{n42} and \eqref{n47}. Adding \eqref{n45} and \eqref{n46}, we obtain \eqref{103.1}.

Now consider the general case \eqref{103.2}. 
Let $\imath_1:Z=s_1^{-1}(0) \to Y$ be the inclusion map. 
Let $\rho_1:Y_1\to Y$ be the blowup of $Y$ along $Z$ and $D_1$ denote the exceptional divisor of $\rho_1$. 
By \cite[Proposition 2.3.6]{Kresch}, we have a surjection
$$A_*(Y_1)\oplus A_*(Z)\lra A_*(Y), \quad (\alpha,\beta)\mapsto {\rho_1}_*\alpha+{\imath_1}_*\beta.$$

For $\beta \in A_* (Z)$, we have
$$\begin{aligned}\sqrt{e}(F,s)\imath_{1 *}\beta &= \sqrt{e}(F|_{Z},s_2)\beta = \sqrt{e}(K^{\perp}/K)e(K|_{Z},s_2)\beta\\
&=e(K|_Z,s_2)\sqe(K^\perp/K)\beta\\
&=e(K|_Z,s_2)\sqe(K^\perp/K,s_1){\imath_1}_*\beta\end{aligned}$$
by \eqref{102.1}, \eqref{103.1} and \eqref{n47}. Hence \eqref{103.2} holds for ${\imath_1}_*\beta$. 

It remains to prove 
\beq\label{n48} \sqrt{e}(F,s){\rho_1}_*\alpha=e(K|_Z,s_2)\sqe(K^\perp/K,s_1){\rho_1}_*\alpha \quad \text{for }\alpha\in A_*(Y_1).\eeq
As localized Euler classes and localized square root Euler classes are bivariant classes, 
\eqref{n48} follows from
\beq\label{n49}
\sqrt{e}(F|_{Y_1},s|_{Y_1})\alpha=e(K|_{D_1},s_2|_{D_1})\sqe(K^\perp/K|_{Y_1},s_1|_{Y_1})\alpha.\eeq

To prove \eqref{n49},
let $L_1=\sO_{Y_1}(D_1)$ which is an isotropic subbundle of $K^{\perp}/K|_{Y_1}$. Then  $L_1=\tL_1/(K|_{Y_1})$ for some isotropic subbundle $\tL_1$ of $F|_{Y_1}$, which is contained in $K^{\perp}|_{Y_1}$. Obviously, we have $\tL_1^{\perp}/\tL_1\cong L_1^{\perp}/L_1$. 
The left hand side of \eqref{n49} is 
$$\sqrt{e}(F|_{Y_1},s|_{Y_1})\alpha =\sqrt{e}(\tL_1^{\perp}/\tL_1)e(\tL_1,s|_{Y_1})\alpha=\sqrt{e}(L_1^{\perp}/L_1)e(\tL_1,s|_{Y_1})\alpha$$ 
by \eqref{103.1} because $s|_{Y_1} \in H^0(Y_1,\tL_1)$. 
The right hand side of \eqref{n49} is
$$\begin{aligned}
e(K|_{D_1},s_2|_{D_1})&\sqe(K^\perp/K|_{Y_1},s_1|_{Y_1})\alpha\\
&=e(K|_{D_1},s_2|_{D_1})\sqrt{e}(L_1^{\perp}/L_1)e(L_1,s_1|_{Y_1})\alpha\\
&=\sqrt{e}(L_1^{\perp}/L_1)e(K|_{D_1},s_2|_{D_1})e(L_1,s_1|_{Y_1})\alpha.
\end{aligned}$$
Hence it suffices to prove 
\beq\label{n50}
e(\tL_1,s|_{Y_1})=e(K|_{D_1},s_2|_{D_1})\circ e(L_1,s_1|_{Y_1}).\eeq

To prove \eqref{n50}, let $X_1=X\times_Y Y_1$ and  consider the diagram
$$\xymatrix{
X_1 \ar[r] \ar[d] & D_1 \ar[r] \ar[d]^{s_2|_{D_1}} & Y_1 \ar[d]^{s|_{Y_1}}\\
Y_1 \ar[r]^{0_{K|_{Y_1}}} & K|_{Y_1} \ar[r]^\kappa \ar[d] & \tL_1 \ar[d] \\
& Y_1 \ar[r]^{0_{L_1}} & L_1
}$$
of Cartesian squares, where the bottom right vertical arrow is the quotient map $\tL_1\to \tL_1/(K|_{Y_1})=L_1$. Note that the normal bundle of $\kappa$ is $N_{K|_{Y_1}/\tL_1}=L_1|_{K|_{Y_1}}$ so that 
$$\kappa^!=0_{L_1}^!:A_*(Y_1) \to A_{*-1}(D_1).$$ 
By the functoriality (cf. \cite[Chapter 6]{Ful} or \cite[\S5.1]{Kresch}) for $\kappa\circ 0_{K|_{Y_1}}=0_{\tL_1}$, 
we have the equality 
$$0_{\tL_1}^!=0^!_{K|_{Y_1}}\circ \kappa^!=0_{K|_{D_1}}^! \circ 0_{L_1}^!$$ 
which implies \eqref{n50} by \eqref{n33}.
\end{proof}

We thus proved the following.
\begin{theo}\label{n51}
Let $(F,q,or)$ be an $SO(2n)$-bundle over a \DM stack $Y$ and $s\in H^0(F)$ be an isotropic section whose zero locus is denoted by $X$. 
The localized square root Euler class $\sqe(F,s)\in A_X^n(Y)$ in Definition \ref{n39} is a well defined bivariant class which satisfies
$$\imath_*\circ \sqe(F,s)=\sqe(F).$$
Moreover, $\sqe(F,s)$ coincides with the Oh-Thomas class in \cite{OhTh} constructed by degeneration to the normal cone and cosection localized Gysin map. 
\end{theo}
The last statement will be proved in \S\ref{SOT}.

\medskip

An immediate application of Lemma \ref{103} is the following (cf. \cite[Proposition 3.8]{OhTh} for the Oh-Thomas class).
\begin{coro}\label{105}
Let $F$ be an $SO(2n)$-bundle over a \DM stack $X$, and $K$ be an isotropic subbundle. Let $C$ be an isotropic subcone of $F$. Suppose that the cone $C$ is invariant under the action of $K$ on $F$
and $K \subseteq C \subseteq K^{\perp}$. Let $\tau \in H^0(C,F|_C)$, $\tau_1 \in H^0(C/K, K^{\perp}/K)$ be the tautological sections. Then we have
\begin{equation}\label{105.1}
\sqrt{e}((K^{\perp}/K)|_{C/K},\tau_1) = \sqrt{e}(F|_C,\tau) \circ p^* 
\end{equation}
where $p:C \to C/K$ is the projection.
\end{coro}

\begin{proof}
Consider the Cartesian square
$$\xymatrix{
K \ar[r] \ar[d]_{\pi_K} & C \ar[d]^p \\
X \ar[r]^0 & C/K.
}$$
Since the localized square root Euler class commutes with smooth pullbacks by Lemma \ref{102} (2), we have
\begin{equation}\label{105.2}
\pi_K^* \circ \sqrt{e}(K^{\perp}/K,\tau_1) = \sqrt{e}(K^{\perp}/K,\tau_1) \circ p^*:A_*(C/K)\lra A_{*-n}(X).    
\end{equation}
Let $\tau_K \in H^0(K,K|_K)$ be the tautological section. Applying $0_K^!=e(K,\tau_K) $ to \eqref{105.2}, we obtain \eqref{105.1} by \eqref{103.2} in Lemma \ref{103}.
\end{proof}

We end this section with the following version of the Whitney sum formula.
\begin{lemm}\label{107}
Let $F$ and $F'$ be special orthogonal bundles over a \DM stack $Y$, and $s,s'$ be isotropic sections of $F$ and $F'$, respectively. Then we have
\begin{equation}\label{107.1}
\sqrt{e}(F\oplus F',(s,s'))=\sqrt{e}(F',s')\circ \sqrt{e}(F,s).
\end{equation}
\end{lemm}
\begin{proof}
Note that both sides of \eqref{107.1} commute with projective pushforwards. If $s'=0$, then \eqref{107.1} follows directly from the definition \eqref{n41}. In particular, 
\eqref{107.1} holds for $\xi \in A_* (X')$ where $X'={s'}^{-1}(0)$. 
Hence, by Lemma \ref{102}, we may replace $Y$ by its blowup along $X'$ and assume that $X'$ is a divisor of $Y$. There is a line bundle $L'$ which is an isotropic subbundle of $F'$ such that $s' \in H^0(Y,L')$. 

Similarly, we may replace $Y$ by its blowup along $X=s^{-1}(0)$ and assume that $X$ is a divisor of $Y$. Let $L=\sO_Y(X)$ be the line bundle of the divisor $X$, which is an isotropic subbundle of $F$. 
Then $L \oplus L' $ is an isotropic subbundle of $F \oplus F'$. By Lemma \ref{103}, we have 
\begin{align*}
\sqrt{e} (F \oplus F',(s,s')) &= \sqrt{e}((L\oplus L')^{\perp}/(L\oplus L'))e(L\oplus L',(s,s'))\\
&= \sqrt{e}((L^{\perp}/L) \oplus (L'^{\perp}/L')) e(L,s) e(L',s') \\
&= \sqrt{e}(L^{\perp}/L) \sqrt{e}(L'^{\perp}/L') e(L,s) e(L',s') \\
&= \sqrt{e}(F,s)\sqrt{e}(F',s')    
\end{align*}
as desired. Note that $e(L\oplus L',(s,s'))=e(L,s) e(L',s')$ follows from the functoriality of the Gysin homomorphism in \cite{Ful}. 
\end{proof}

\bigskip

\section{Comparison with the Oh-Thomas construction}\label{SOT}

In this section, we recall the Oh-Thomas construction of localized square root Euler class and prove that it equals \eqref{n40}. The results of this section will not be used in the rest of this paper. 

\subsection{Localized Edidin-Graham class by Oh-Thomas}\label{S2.2}

%

In this subsection, we recall the Oh-Thomas construction from \cite{OhTh}. 

Let $F$ be an $SO(2n)$-bundle over a \DM stack $Y$ and 
let $s$ be an isotropic section of $F$. 
In the special case where $F$ admits a positive maximal isotropic subbundle $V$ and $s$ is a section of $V$, 
$\sqe(F)=e(V)$ is localized to $X=s^{-1}(0)$ by 
\beq\label{68} 
\sot(F,s,V)=e(V,s)=(-1)^n0^!_{V^\vee,s^\vee}(0_{V^\vee})_*:A_*(Y)\lra A_{*-n}(X)\eeq
where $e(V,s)$ is the localized Euler class in \S\ref{Sn3.2} and $0^!_{V^\vee,s^\vee}$ denotes the cosection localized Gysin map in \S\ref{Sn4.2}. The second equality is \eqref{n52}.  

\medskip

When $s$ is not a section of $V$, Oh-Thomas use the degeneration of $Y$ to the normal cone 
$$C_{Z/Y}\subset V^\vee|_{Z}$$ to the zero locus $Z=(s_1)^{-1}(0)$ of  
\beq\label{n60} s_1:\sO_Y\lra F\lra V^\vee\eeq
where the second arrow is the surjection in \eqref{63}. 
Let $s_2=s|_{Z}\in H^0(Z,V)$ whose dual is a cosection   
$$s_2^\vee:V^\vee|_{Z}\lra \sO_{Z}.$$
Then the zero locus of $s_2$ in $Z$ is precisely the zero locus $X=s^{-1}(0)$ of $s$. 
Following the notation of \S\ref{Sn4.2}, we let 
$$V|_{Z}^\vee(s_2^\vee)=V|_X^\vee\cup \ker(s_2^\vee|_{Z-X}).$$
Since $s$ is isotropic, we have the inclusion 
$$C_{Z/Y}\subset V|_{Z}^\vee(s_2^\vee). $$
The Oh-Thomas class is defined as the composition 
\beq\label{62}
\sot(F,s,V):A_*(Y)\mapright{\sp} A_*(C_{Z/Y})\lra A_*(V|^\vee_{Z}(s_2^\vee))\xrightarrow{ (-1)^n0^!_{V|^\vee_{Z},s_2^\vee} } A_{*-n}(X)
\eeq 
where $\sp$ denotes the specialization map and $0^!_{V|^\vee_{Z},s_2^\vee}$ is the cosection localized Gysin map in \S\ref{Sn4.2}. 

It is straightforward from the definition that $\sot(F,s,V)$ commutes with smooth pullbacks, i.e. if $g:Y'\to Y$ is smooth with the induced map $h:X'=X\times_YY'\to X$, 
we have the equality 
\beq\label{n62} h^*\circ \sot(F,s,V)=\sot(g^*F,g^*s,g^*V)\circ g^*.\eeq

\medskip

In the general case where $F$ may not admit a positive maximal isotropic subbundle, we use the tower \eqref{55} and the identity \eqref{57}. 
By \eqref{59}, $F|_Q$ admits a positive maximal isotropic subbundle $\Lambda$. 
Using the special case \eqref{62}, Oh-Thomas define the following.
\begin{defi}\label{n54} \cite[\S3]{OhTh} With the notation above, the Oh-Thomas class is defined as   
\beq\label{64}
\sot(F,s):A_*(Y)\lra A_{*-n}(X), 
\eeq
$$\sot(F,s)(\xi)=(p|_{Q\times_YX})_*\left(\sot(F|_Q,s|_Q,\Lambda) h\cap p^*(\xi)\right).$$
\end{defi}

%

\medskip

\subsection{The Oh-Thomas class equals \eqref{n40}}

In this subsection, we prove that the Oh-Thomas class \eqref{64} equals \eqref{n40}.

\begin{theo}\label{n55}
For an $SO(2n)$-bundle $F$ and an isotropic section $s$, we have the equality
$$\sqe(F,s)=\sot(F,s).$$
In particular, the Oh-Thomas class $\sot(F,s)$ is bivariant and equals $\sot(F,s,V)$ whenever $F$ admits a positive maximal isotropic subbundle $V$. 
\end{theo}
The rest of this section is devoted to a proof of Theorem \ref{n55}.

We first reduce the proof to the special case.
\begin{lemm}\label{n56}
Theorem \ref{n55} holds if the special case \eqref{62} equals $\sqe(F,s)$ when $F$ admits a positive maximal isotropic subbundle. 
\end{lemm}
\begin{proof}
By \eqref{64} and Lemma \ref{102} (1), we have  
$$\begin{aligned}
\sot(F,s)\xi&=(p|_{Q\times_YX})_*\left(\sot(F|_Q,s|_Q,\Lambda) h\cap p^*(\xi)\right)\\
&=(p|_{Q\times_YX})_*\left(\sqrt{e}(F|_Q,s|_Q) h\cap p^*(\xi)\right)\\
&=\sqrt{e}(F,s)p_*\left( h\cap p^*(\xi)\right)=\sqe(F,s)\xi
\end{aligned}$$
because $\sot(F|_Q,s|_Q,\Lambda)=\sqrt{e}(F|_Q,s|_Q)$ by assumption. 
\end{proof}

By this lemma, we may assume from now on that $F$ admits a positive maximal isotropic subbundle $V$ and it suffices to prove 
\beq\label{n58}
\sqe(F,s)=\sot(F,s,V)=(-1)^n0^!_{V|^\vee_{Z},s_2^\vee} \circ \sp.\eeq
Actually we may further assume that 
\beq\label{n63} F=V\oplus V^\vee.\eeq 
Indeed, let $S$ be the inverse image of $$\id\in \mathrm{Hom}(V^\vee,V^\vee)=H^0(Y,\cH om(V^\vee, V^\vee))$$ by  
the homomorphism of locally free sheaves 
$$\cH om(V^\vee, F)\lra \cH om(V^\vee, V^\vee)$$
on $Y$ induced by the surjection $F\to V^\vee$ in \eqref{63}. 
Then $\pi:S\to Y$ is an affine bundle of rank $n^2$.
By \eqref{n62} and Lemma \ref{102} (2), if we let $\nu:S\times_YX\to X$ denote the restriction of $\pi$, we find that 
$$\nu^*\circ\sqe(F,s)=\sqe(\pi^*F,\pi^*s)\circ\pi^*, $$ 
$$\nu^*\circ \sot(F,s,V)=\sot(\pi^*F,\pi^*s,\pi^*V)\circ \pi^*.$$
By \cite[Corollary 2.5.7]{Kresch}, $\pi^*$ and $\nu^*$ are isomorphisms and hence
$$\sqe(\pi^*F,\pi^*s)=\sot(\pi^*F,\pi^*s,\pi^*V)\quad \Longrightarrow \quad 
\sqe(F,s)=\sot(F,s,V).$$
Note that $\pi^*F$ splits as $\pi^*V\oplus \pi^*V^\vee$. 
Therefore, we may assume that $F$ is the direct sum \eqref{n63} and 
\beq\label{n64} s=(s_2,s_1),\quad s_2\in H^0(V),\ \ s_1\in H^0(V^\vee).\eeq

\medskip

We next reduce the proof of \eqref{n58} 
to the cone case $C_{Z/Y}$. 
\begin{lemm}\label{n57}
Under \eqref{n63} and \eqref{n64}, letting $C=C_{Z/Y}$, 
$$\sqe(F,s)=\sqe(F|_C,(s_2|_C,\tau))\circ \sp$$
where $\tau$ is the tautological section by the inclusion $C\subset V|_Z^\vee$. 
Hence \eqref{n58} holds if 
\beq\label{n59} \sqe(F|_C,(s_2|_C,\tau))=(-1)^n0^!_{V|^\vee_{Z},s_2^\vee|_Z}:A_*(C)\lra A_{*-n}(X).\eeq
\end{lemm}
\begin{proof}
Let 
$$M=M_{Z/Y}^\circ\lra \PP^1\times Y\lra \PP^1$$ be the deformation space of $Y$ to the normal cone $C$ of $Z$ in $Y$ (cf. \cite[Chapter 5]{Ful}). 
Let $\lambda$ denote the coordinate for $\PP^1$ and
let $\imath_\lambda : \{\lambda\}\to \PP^1$ be the inclusion map.  
Then the fiber $M|_\infty$ over $\lambda=\infty$ is $C$ and $M\times_{\PP^1}(\PP^1-\infty)\cong Y\times \bbA^1.$
By the definition of $\sp$, for $\xi\in A_*(Y)$ and any extension  $\tilde{\xi}\in A_{*+1}(M)$ of $\xi|_{Y\times \bbA^1}$, we have 
$$\imath_{\infty}^!\tilde{\xi}=\sp(\xi),\quad \imath_1^!\tilde{\xi}=\xi.$$ 

By the graph construction \cite[Remark 5.1.1]{Ful}, we have an embedding $M \hookrightarrow V\dual \times \PP^1$. Let
$$\tilde{\tau} \in H^0(M,V\dual|_M)$$
be the tautological section induced by the embedding. Then the fibers over $\lambda \in \PP^1$ are 
$$\tilde{\tau}|_{\lambda} =
\begin{cases} 
\lambda \cdot s_1 \in H^0(Y, V\dual) &  \text{ if } \lambda \neq \infty\\
\tau \in H^0(C, V\dual|_C) &  \text{ if } \lambda = \infty.
\end{cases}$$
and 
$$\tilde{s}=(s_2|_M,\tilde{\tau}) \in H^0(M,F|_{M})$$
is an isotropic section whose zero locus is $X \times \PP^1$.

For $\xi \in A_* (Y)$ and any extension $\tilde{\xi}\in A_{*+1}(M)$ of $\xi|_{Y\times \bbA^1}$,  
since $$\imath_1^!=\imath_\infty^! : A_{*+1}(X(s) \times \PP^1) \lra A_{*} (X(s)),$$
we have
$$\begin{aligned} 
\sqrt{e}(F,s)\xi & = \sqe(F,s)\imath_1^!\tilde{\xi}=\imath_1^!
\sqrt{e}(F|_M, \tilde{s}) \tilde{\xi}\\
& = \imath_\infty^!\sqrt{e}(F|_M,\tilde{s}) \tilde{\xi} = \sqrt{e}(F|_C,(s_2|_C,\tau))\imath_\infty^!\tilde{\xi} \\
&=\sqrt{e}(F|_C,(s_2|_C,\tau))\mathrm{sp}(\xi)
\end{aligned}$$
by Lemma \ref{102} (3). 
\end{proof}

The proof of Theorem \ref{n55} is complete if we prove the following. 
\begin{lemm}\label{n65}
\eqref{n59} holds under the assumptions of Lemma \ref{n57}.
\end{lemm}
\begin{proof}
Let $\tZ\to Z$ be the blowup of $Z$ along $X$ and $D$ be the exceptional divisor so that we have the Cartesian square
\[\xymatrix{
D\ar[r]^\jmath \ar[d]_{\rho'} & \tZ\ar[d]^\rho\\
X\ar[r]^\imath & Z.}\]
The cosection $s_2^\vee:V|^\vee_Z\to \sO_Z$ pulls back to a surjection
$$s_2^\vee|_{\tZ}:V|_{\tZ}^\vee\lra \sO_{\tZ}(-D)$$
whose kernel is denoted by $\bar V^\vee$. Taking the dual, we have an exact sequence
$$0\lra \sO_{\tZ}(D)\lra V|_{\tZ}\lra \bar V\lra 0.$$  
By \cite[Proposition 2.3.6]{Kresch} (cf. \cite[Example 1.8.1]{Ful}), the Cartesian square
\[\xymatrix{
C|_D\cap\bar V^\vee\ar[r]^{\nu'}\ar[d]_{\mu'} & C|_{\tZ}\cap\bar V^\vee\ar[d]^\mu\\
C|_X\ar[r]^{\nu} & C} \]
gives us the exact sequence
$$A_*(C|_D\cap\bar V^\vee)\lra A_*(C|_{\tZ}\cap\bar V^\vee)\oplus A_*(C|_X)\xrightarrow{(\mu_*,\nu_*)} A_*(C)\lra 0.$$
Hence it suffices to check \eqref{n59} for 
$\mu_*\alpha$ and $\nu_*\beta$ with $\alpha\in A_*(C|_{\tZ}\cap\bar V^\vee)$ and $\beta\in A_*(C|_X)$.  
By the definition \eqref{n38}, we have
\beq\label{n66}
0^!_{V|^\vee_{Z},s_2^\vee|_Z}\mu_*\alpha=-\rho'_*\jmath^*0^!_{\bar V^\vee}\alpha, \quad
0^!_{V|^\vee_{Z},s_2^\vee|_Z}\nu_*\beta=0^!_{V|^\vee_X}\beta.
\eeq

For $\beta\in A_*(C|_X)$, Lemma \ref{102} (1), Lemma \ref{103} and \eqref{n14.1} imply
$$\begin{aligned}
\sqrt{e}(F|_{C}, (s_2|_C,\tau)) \nu_*\beta &= \sqrt{e}(F|_{C|_X}, (0,\tau|_{C|_X}))\beta\\ 
&=(-1)^ne(V|^\vee_{C|_X},\tau|_{C|_X}) \beta \\
&= (-1)^n0^!_{V|^\vee_X}\beta= (-1)^n0^!_{V|^\vee_{Z},s_2^\vee|_Z}\nu_*\beta.\end{aligned}$$

It remains to prove 
\beq\label{n67}\sqrt{e}(F|_{C}, (s_2|_C,\tau)) \mu_*\alpha= (-1)^n0^!_{V|^\vee_{Z},s_2^\vee|_Z}\mu_*\alpha.\eeq
By \eqref{n66}, the right hand side of \eqref{n67} equals 
\beq\label{n70}(-1)^{n-1}\rho'_*\jmath^*0^!_{\bar V^\vee}\alpha 
=(-1)^{n-1}\rho'_* (0_{\bar V^\vee}\circ \jmath)^!\alpha.\eeq
By Lemma \ref{102} (1), the left hand side of \eqref{n67} equals 
\beq\label{n71}\rho'_*\sqrt{e}(F|_{C|_{\tZ}\cap \bar V^\vee}, (s_2|_{C|_{\tZ}\cap \bar V^\vee},\tau|_{C|_{\tZ}\cap \bar V^\vee}) )\alpha.\eeq
Since $(s_2|_{C|_{\tZ}\cap \bar V^\vee},\tau|_{C|_{\tZ}\cap \bar V^\vee})^{-1}(0)=D$ and the normal bundle of $D$ in $\bar V^\vee$ is 
$$N_{D/\bar V^\vee}=\sO_D(D)\oplus \bar V|_D^\vee$$ which is maximal isotropic with sign $(-1)^{n-1}$ (cf. \S\ref{Sorth}), by Lemma \ref{n68} below,
we then have 
\beq\label{n69}\sqrt{e}(F|_{C|_{\tZ}\cap \bar V^\vee}, (s_2|_{C|_{\tZ}\cap \bar V^\vee},\tau|_{C|_{\tZ}\cap \bar V^\vee}) )\alpha=(-1)^{n-1}(0_{\bar V^\vee}\circ \jmath)^!\alpha.\eeq
The equality \eqref{n67} follows from \eqref{n70}, \eqref{n71} and \eqref{n69}.
\end{proof}

\begin{lemm}\label{n68}
Let $F$ be an $SO(2n)$-bundle over a \DM stack $Y$ and $s$ be an isotropic section. If the inclusion map $\imath: X=s^{-1}(0) \to Y$ is a local complete intersection morphism, then $N=N_{X/Y}$ is an isotropic subbundle of $F|_{X}$ and
\begin{equation}\label{104.1}
\sqrt{e}(F,s)=\sqrt{e}(N^{\perp}/N)\circ \imath^!:A_*(Y) \to A_{*-n}(X).    
\end{equation}
\end{lemm}

\begin{proof}
In the proof of  \cite[Proposition 4.3]{OhTh}, it was proved that the normal cone $C_{X/Y}=N_{X/Y}$ is an isotropic subbundle of $F|_{X}$. 
Indeed, by MacPherson's graph construction \cite[Remark 5.1.1]{Ful}, the deformation space $M=M^\circ_{X/Y}$ is the closure of the embedding
\begin{equation}\label{104.2}
Y \times \bbA^1 \to F \times \PP^1 : (y,\lambda) \mapsto (\lambda \cdot s(y), [1:\lambda]).    
\end{equation}
Let $Q \subseteq F$ be the quadric cone of isotropic vectors in the fibers of $F$. Since $\lambda\cdot s \in H^0(Y\times \bbA^1,F\times \bbA^1)$ is an isotropic section, the image of the embedding \eqref{104.2} is contained in $Q \times \PP^1$. Hence the closure $M$ is also contained in $Q \times \PP^1$. Considering the fiber over $\infty \in \PP^1$, we deduce $M_{\infty}=N \subseteq Q$. Hence $N \subseteq Q|_{X}$, which means that $N$ is isotropic.

Let us use the notation in Definition \ref{n39}. 
By the excess intersection formula, 
we have $$\imath^! = e(\rho^*N/L)\circ\jmath^*:A_*(\tY) \lra A_{*-m}(D)$$
where $L=\sO_{\tY}(D)$. By \eqref{n41} and \eqref{n42}, we have 
\begin{align*}
\sqrt{e}(F,s)(\rho_* \alpha + \imath_* \beta) & = \hat{\rho}_* \sqrt{e}(\tF)\jmath^*\alpha + \sqrt{e}(F) \beta \\
& = \hat{\rho}_*  \sqrt{e}(N^{\perp}/N)e(\rho^*N/L)\jmath^* \alpha + \sqrt{e}(N^{\perp}/N)e(N)\beta\\
& = \sqrt{e}(N^{\perp}/N) \hat{\rho}_* \imath^! \alpha + \sqrt{e}(N^{\perp}/N)\imath^!\imath_*\beta\\
& = \sqrt{e}(N^{\perp}/N)  \imath^! (\rho_* \alpha + \imath_*\beta)
\end{align*}
as desired.
\end{proof}
This completes our proof of Theorem \ref{n55}.

\bigskip

\section{Further localization of square root Euler class by extra isotropic section}\label{Scos}

By Theorem \ref{n51}, the square root Euler class $\sqe(F)$ of an $SO(2n)$-bundle $F$ is localized to the zero locus $X$ of an isotropic section $s$. In this section, we show that we can further localize $\sqe(F)$ if there is an extra isotropic section $t$ with $s\cdot t=0$.

We use the notation of Definition \ref{n39}. 
Let $F$ be an $SO(2n)$-bundle over a \DM stack $Y$, and $s,t$ be two isotropic sections of $F$ such that $s \cdot t=0$. Let $X=s^{-1}(0)$ and $X(t)=X\cap t^{-1}(0)$.

Let $\rho:\tY\to Y$ be the blowup of $Y$ along $X$ and $D$ be the exceptional divisor.
Then $L=\sO_{\tY}(D)$ and $\tF=L^{\perp}/L$ be the induced $SO(2n-2)$-bundle. 
Since $s\cdot t=0$, the isotropic section $t|_{\tY}=\rho^*t$ of $F|_{\tY}=\rho^*F$ 
induces an isotropic section $\tit$ of $\tF$. Let
\begin{equation}\label{111.3}
X(t)^\# := \rho'\Big(D \cap \tit^{-1}(0)\Big) \cup X(t)\subset X,
\end{equation}
where $\rho'=\rho|_D$. Consider the commutative diagram
$$\xymatrix{
D (t) \ar[r]^{\jmath''} \ar[d]_{\rho'''} 
& D(\tit) \ar[r]^{\jmath'} \ar[d]_{\rho''} & D \ar[r]^{\jmath} \ar[d]^{\rho'} & \tY \ar[d]^{\rho}\\
X(t) \ar[r]^{\imath''} & X(t)^\# \ar[r]^{\imath'} & X \ar[r]^{\imath} & Y.
}$$ 
\begin{defi}\label{n72} With the above notation,
the {\em square root Euler class localized by} $s$ and $t$ is defined as   
\begin{equation}\label{111.2}
\sqrt{e}(F,s;t) : A_* (Y) \longrightarrow A_{*-n} (X(t)^\#)    
\end{equation}
\begin{equation}\label{111.1}\begin{aligned}
    \sqrt{e}(F,s;t) (\rho_*\alpha + \imath_* \beta) &= \rho''_* \jmath^! \sqrt{e}(\tF,\tit) \alpha + \imath''_* \sqrt{e} (F,t) \beta\\
    & = \rho''_* \sqrt{e}(\tF,\tit) \jmath^* \alpha + \imath''_* \sqrt{e} (F,t) \beta \end{aligned}
\end{equation}
for $\alpha\in A_*(\tY)$ and $\beta\in A_*(X)$. 
\end{defi}

We first prove that $\sqrt{e}(F,s;t)$ in \eqref{111.1} is well-defined.

\begin{lemm}\label{112}
$\sqrt{e}(F,s;t)(\xi)$ in \eqref{111.1} is independent of the choice of the decomposition $\xi= \rho_* \alpha + \imath_* \beta$.
\end{lemm}
\begin{proof} As in the proof of Lemma \ref{n43}, 
it suffices to show that
\begin{equation}\label{112.1}
\rho''_*\jmath^*\sqrt{e}(\tF,\tit) \jmath_* \gamma = \imath''_*\sqrt{e}(F,t)\rho'_*\gamma    
\end{equation}
holds for any $\gamma \in A_* (D)$. 
Since $\sqrt{e}(\tF,\tit)$ commutes with projective pushforwards by Lemma \ref{102} (1), the left hand side of \eqref{112.1} is equal to
$$\rho''_* e(L)\sqrt{e}(\tF,\tit)\gamma = \rho''_* \jmath''_* e(L,t|_{\tit^{-1}(0)})\sqrt{e}(\tF,\tit)\gamma.$$
By Lemma \ref{103}, we have
$$ \rho''_* \jmath''_*e(L,t|_{\tit^{-1}(0)})\sqrt{e}(\tF,\tit)\gamma = \rho''_* \jmath''_*\sqrt{e}(F,t)\gamma= \imath''_* \rho'''_* \sqrt{e}(F,t)\gamma.$$
Again by Lemma \ref{102} (1), we have 
$$\imath''_* \rho'''_*  \sqrt{e}(F,t)\gamma= \imath''_*\sqrt{e}(F,t) \rho'_*\gamma,$$
which proves \eqref{112.1}.
\end{proof}

Next we show that $\sqe(F,s;t)$ is a further localization of $\sqe(F,s)$. 
\begin{lemm}\label{113}
\begin{equation}\label{113.1}
\imath'_* \circ \sqrt{e}(F,s;t)=  \sqrt{e}(F,s):A_*(Y) \to A_{*-n}(X).    
\end{equation}
\end{lemm}
\begin{proof}
Since the localized square root Euler class $\sqrt{e}(\tF,\tit)$ commutes with the refined Gysin pullback by Lemma \ref{102} (3), for $\alpha \in A_*(\tY)$, we have 
\begin{align*}
\imath'_*\rho''_*\jmath^!\sqrt{e}(\tF,\tit) \alpha  = \rho'_*\jmath'_* \sqrt{e}(\tF,\tit)\jmath^*\alpha  = \rho'_* \sqrt{e}(\tF) \jmath^* \alpha.
\end{align*}
On the other hand, for $\beta \in A_* X$, we have 
$$\imath'_*\imath''_* \sqrt{e}(F,t)\beta = \sqrt{e}(F)\beta.$$
This completes the proof.
\end{proof}

\medskip

Unlike $X(t)=X\cap t^{-1}(0)$, $X(t)^\#$  in \eqref{111.3} is not so easy to work with.
To simplify $X(t)^\#$, we need a linear independence condition for $s$ and $t$. 
\begin{defi}\label{n73}
Let $s$ and $t$ be sections of a vector bundle $F$ over a \DM stack $Y$.
We say $s$ and $t$ are \emph{independent} away from a closed substack $Z$ 
if for every $y\in Y-Z$, there are an \'etale neighborhood $U$ of $y$ on $Y-Z$ and a splitting 
$$F|_{U}\cong F'\oplus F''$$
of vector bundles 
such that $$s|_{U}\in H^0(F') \and t|_{U}\in H^0(F'').$$ 
\end{defi}

Now we can state the main result of this section. 

\begin{theo}\label{n74}
Let $F$ be an $SO(2n)$-bundle over a \DM stack $Y$, and $s, t$ be isotropic sections of $F$
satisfying $s \cdot t=0$. Let $X=s^{-1}(0)$. 
If $s$ and $t$ are independent away from a closed substack $Z$ containing $t^{-1}(0)$,  
then we have a homomorphism 
\beq\label{n76}  \sqe(F,s;t):A_*(Y)\lra A_{*-n}(X\cap Z) \eeq
defined by \eqref{111.1}. 
Moreover, $\sqe(F,s;t)$ is a bivariant class in $A_{X\cap Z}^n(Y)$ satisfying
\beq\label{n75} \hat{\imath}_*\circ \sqe(F,s;t)=\sqe(F,s)\eeq
where $\hat{\imath}:X\cap Z\to X$ denotes the inclusion. 
\end{theo}
\begin{proof}
The proof of bivariance is similar to that of Lemma \ref{102} and we omit it.
The equality \eqref{n75} is a direct consequence of Lemma \ref{113} by \eqref{n77}. 
So we only have to prove the inclusion of the reduced schemes 
\beq\label{n77}  X(t)^\#_\redd\subset (X\cap Z)_\redd \eeq
when $t$ is independent of $s$ away from $Z$. 

By \eqref{111.3}, it suffices to show that $\rho'(D\cap \tit^{-1}(0))_\redd\subset (X\cap Z)_\redd$, which is a direct consequence of  
\beq\label{n78} \tit^{-1}(0)_\redd\subset (\tY\times_Y Z)_\redd,\eeq
i.e. $\tit|_{\tY\times_Y(Y-Z)}$ is nowhere vanishing. 
Clearly \eqref{n78} is local and hence we may assume that $Y=U$ for the $U$ in Definition \ref{n73} 
and that $F$ splits as 
\beq\label{n79} F\cong  F'\oplus F''\eeq 
with $s\in H^0(F')$ and $t\in H^0(F'')$.

As $L=\sO_{\tY}(D)\subset F|_{\tY}$ 
is the dual of the surjection $$s^\vee:F|_{\tY}^\vee\lra \sO_{\tY}(-D)$$ 
which factors through $F'|^\vee_{\tY}$ because $s\in H^0(F')$, we find that 
$$L|_{\tY}\subset F'|_{\tY}.$$ 
Since $t$ is a nowhere vanishing section of $F''$ by $Z\supset t^{-1}(0)$ while $L|_{\tY}$ is a subbundle of $F'|_{\tY}$, the section $\tit$ of 
$$\begin{aligned}
\tF&=L^\perp/L\subset F|_{\tY}/L = F'|_{\tY}/L\oplus F''|_{\tY}\end{aligned}$$ 
induced by $t$ is nowhere vanishing on $\tY$. This proves \eqref{n78} and hence the theorem.
\end{proof}

In the same spirit, we can further localize $e(V,s)$  in \eqref{n14} when there is an extra section.

\begin{theo}\label{n81}
Let $V$ be a vector bundle of rank $r$ over a \DM stack $Y$. Let $s$ and $t$ be sections of $V$, independent away from a closed substack $Z\supset t^{-1}(0)$. Let  $X=s^{-1}(0)$. 
Let $\rho:\tY\to Y$ be the blowup of $Y$ along $X$ so that we have a Cartesian diagram
\[ \xymatrix{
D_Z\ar[r]^{\jmath'} \ar[d]_{\rho''} & D\ar[d]^{\rho'}\ar[r]^{\jmath} & \tY\ar[d]^\rho\\
X\cap Z\ar[r]^{\imath'} & X\ar[r]^{\imath} &Y
}\]
and the exact sequence \eqref{n32}. 
Then we have a homomorphism 
\beq\label{n76a}  e(V,s;t):A_*(Y)\lra A_{*-r}(X\cap Z) \eeq
defined by 
\[ e(V,s;t)(\rho_*\alpha+\imath_*\beta)=\rho''_*e(\bar V,\bar{t}) \jmath^*\alpha+e(V,t)\beta,\quad \forall 
\alpha\in A_*(\tY), \beta\in A_*(X)\] 
where 
$\bar V=V|_{\tY}/\sO_{\tY}(D)$ and $\bar t\in H^0(\bar V)$ is the section induced by $t$. Moreover, $e(V,s;t)$ is a bivariant class in $A_{X\cap Z}^r(Y)$ satisfying
\beq\label{n75a} \imath'_*\circ e(V,s;t)=e(V,s).\eeq
\end{theo}
The proof of this theorem parallels that of Theorem \ref{n74} and we leave it to the reader. 

\begin{rema}
It is also straightforward to generalize Definition \ref{n73} and Theorems \ref{n74}, \ref{n81} to the case where there are more than two sections inductively. 
\end{rema}

\medskip

We end this section with the following two results which are analogues of Lemma \ref{103} and Corollary \ref{105}. 
\begin{lemm}\label{116}
Let $K$ be an isotropic subbundle of an $SO(2n)$-bundle $F$ over a \DM stack $Y$. Let $s$ and $t$
be isotropic sections of $K^\perp$ with $s\cdot t=0$. 
Let $s_1,t_1 \in H^0(K^{\perp}/K)$ be the isotropic sections of the orthogonal bundle $K^{\perp}/K$ induced by the sections $s,t$, and $s_2\in H^0(K|_{s_1^{-1}(0)\cap Z})$ be the restriction of $s$. 
If $s$ and $t$ (resp. $s_1$ and $t_1$) are independent away from a closed substack $Z\supset t_1^{-1}(0)$, then for any $\xi \in A_*(Y)$,
\begin{equation}\label{116.1}
\sqrt{e}(F,s;t) \xi = e(K,s_2) \sqrt{e}(K^{\perp}/K,s_1;t_1) \xi    
\end{equation}
in $A_*(X\cap Z)$ where $X=s^{-1}(0)$.
\end{lemm}

\begin{coro}\label{117}
Under the assumptions of Corollary \ref{105}, suppose that we have an isotropic section $t \in H^0(K^{\perp})$, such that $t|_C\cdot \tau=0 \in H^0(C,\sO_C)$ where $\tau\in H^0(F|_C)$ denotes the tautological section. Let $t_1 \in H^0(X,K^{\perp}/K)$  be the isotropic section  induced by $t$ and $\tau_1\in H^0(C/K,K^\perp/K)$ be the tautological section. 
If $t|_C$ and $\tau$ (resp. $t_1|_{C/K}$ and $\tau_1$) are independent away from $C|_Z$ for a closed substack $Z\supset t_1^{-1}(0)$, 
then for any $\xi \in A_* (C/K)$, we have 
\begin{equation}\label{117.1}
\sqrt{e}(K^{\perp}/K|_{C/K},\tau_1;t_1) \xi = \sqrt{e}(F,\tau;t) p^* \xi \in A_*( Z).
\end{equation}
\end{coro}


The proofs of Lemma \ref{116} and Corollary \ref{117} are similar to those of Lemma \ref{103} and Corollary \ref{105}, and we omit them.

\bigskip

\section{Virtual cycles for symmetric obstruction theories}\label{S5}


In this section, we recall the Oh-Thomas construction of 
a 3-term locally free symmetric resolution 
\beq\label{1} 
\bbE\udot\cong [B\mapright{d} F\cong F^\vee\mapright{d^\vee} B^\vee]
\eeq
of a symmetric obstruction theory  
\beq\label{2}
\phi:\bbE\udot\lra \bbL_X
\eeq
on a \DM stack $X$, perfect of amplitude $[-2,0]$ from \cite{OhTh}. 
We then see  that a cosection $\bsig:\bbE\ldot\to \sO_X[-1]$ of the obstruction theory $\phi$ induces a lift
\beq\label{3}
\tsig: F\lra  \sO_X
\eeq
of $\sigma=h^1(\bsig):h^1(\bbE\ldot)=Ob_X\to \sO_X$ where $\bbE\ldot=(\bbE\udot)^\vee$.

\medskip

Let $\bbL_X=L_X^{\ge -1}$ denote the cotangent complex of $X$, truncated at $-1$.

\begin{defi}\label{4}
An \emph{obstruction theory} on a \DM stack $X$ is a morphism \eqref{2} in the derived category $D^b(\Coh X)$
such that $h^i(\bbE\udot)=0$ for $i>0$, $h^0(\phi)$ is an isomorphism and $h^{-1}(\phi)$ is surjective. 
We let $\bbE\ldot=(\bbE\udot)^\vee$ denote the dual of $\bbE\ldot$ and 
$h^1(\bbE\ldot)=:Ob_X$ is called the \emph{obstruction sheaf} of $X$ with respect to the obstruction theory $\phi$. 

An obstruction theory $\phi:\bbE\udot\to \bbL_X$ is \emph{perfect} of amplitude $[-a,0]$ if \'etale locally $\bbE\udot$ is isomorphic to a complex of locally free sheaves concentrated in degrees $[-a,0]$. 
An obstruction theory $\phi:\bbE\udot\to \bbL_X$, perfect of amplitude $[-a,0]$, is \emph{symmetric} if there exists an isomorphism
\beq\label{5}
\theta:\bbE\udot \lra \bbE\ldot [a]
\eeq 
satisfying $\theta^\vee[a]=\theta$. 

When $a=2$, a square root $$or:\sO_X\mapright{\cong}\det \bbE\ldot$$ of the isomorphism $\sO_X\cong (\det \bbE\ldot)^{\otimes 2}$ from the isomorphism $$\det\theta:(\det\bbE\ldot)^{-1}\cong \det \bbE\udot \lra \det (\bbE\ldot[2])= \det \bbE\ldot$$ is called an \emph{orientation} of $\phi:\bbE\udot\to \bbL_X$. 
\end{defi}

When $\phi:\bbE\udot\to \bbL_X$ is an obstruction theory, perfect of amplitude $[-1,0]$, a homomorphism
\beq\label{7}
\sigma:Ob_X=h^1(\bbE\ldot)\lra \sO_X
\eeq
induces a morphism 
\beq\label{6}
\bsig: \bbE\ldot \mapleft{\cong} \tau^{\le 1}\bbE\ldot \lra \tau^{\ge 1}\tau^{\le 1}\bbE\ldot =Ob_X[-1]\mapright{\sigma} \sO_X[-1]
\eeq
where $\tau^{\ge 1}$ and $\tau^{\le 1}$ denote the truncation funtors in $D^b(\Coh X)$. 
Conversely, given a morphism $\bsig:\bbE\ldot\to \sO_X[-1]$, we have a homomorphism
\beq\label{6.1}\sigma=h^1(\bsig):Ob_X\lra \sO_X\eeq of coherent sheaves. 
Therefore when $a=1$, giving a morphism \eqref{6} is equivalent to giving a homomorphism \eqref{6.1}.
We are thus led to the following definition of a cosection for an obstruction theory in general.

\begin{defi}\cite[Definition 4.6]{Savv}\label{8}
A \emph{cosection} of an obstruction theory $\phi:\bbE\udot\to \bbL_X$ is a morphism
\beq\label{9}
\bsig:\bbE\ldot\lra \sO_X[-1]
\eeq 
in  $D^b(\Coh X)$.
\end{defi}
A cosection \eqref{9} of an obstruction theory $\phi$ induces a cosection \eqref{6.1} 
of the obstruction sheaf $Ob_X$ of $\phi$. 

\medskip

A symmetric obstruction theory admits a symmetric locally free resolution. 
\begin{prop}\label{10} \cite[\S4]{OhTh}
Let $\phi:\bbE\udot\to \bbL_X$ be a symmetric obstruction theory, perfect of amplitude $[-2,0]$, on $X$.
Then there is a locally free resolution \eqref{1} of $\bbE\udot$ where $F$ is equipped with an isomorphism $q:F\cong F^\vee$ satisfying $q^\vee=q$ such that \eqref{5} is 
\beq\label{12}\xymatrix{
\bbE\udot\ar[d]_{\theta} & B\ar[r]^d\ar@{=}[d] & F\ar[d]^{\hat{q}} \ar[r]^{(\hat{q}\circ d)^\vee} & B^\vee\ar@{=}[d]\\
\bbE\ldot[2] & B\ar[r]^{\hat{q}\circ d} & F^\vee \ar[r]^{d^\vee} & B^\vee .
}\eeq
\end{prop}
\begin{proof}
For the reader's convenience, here is a sketch. First find a locally free resolution 
\beq\label{15} A^{-2}\mapright{d_2} A^{-1}\mapright{d_1} A^0\eeq
of $E\udot$ which admits a chain map
\beq\label{16} \xymatrix{
\bbE\udot\ar[d] & A^{-2}\ar[r]^{d_2} \ar[d]^{\theta_2^\vee} & A^{-1}\ar[r]^{d_1}\ar[d]^{\theta_1} & A^0\ar[d]^{\theta_2}\\
\bbE\ldot[2] & A_0 \ar[r]^{d_1^\vee} & A_1\ar[r]^{d_2^\vee} & A_2
}\eeq
that represents $\theta$, where $\theta_1=\theta_1^\vee$ and $A_i=(A^{-i})^\vee$. 
Let  $B=A_0$ and 
\beq\label{17}
F^\vee=\ker (A_1\oplus A^0\xrightarrow{d_2^\vee\oplus \theta_2} A_2)\cong \coker (A^{-2}\xrightarrow{d_2\oplus \theta_2^\vee} A^{-1}\oplus A_0)= F.
\eeq
The inclusion $A_0\to A^{-1}\oplus A_0$ and the projection to $\coker (d_2\oplus \theta_2^\vee)=F$ give us the homomorphism $d:B\to F$ in \eqref{1}.
\end{proof}

If furthermore we have a cosection $\bsig:\bbE\ldot\to \sO_X[-1]$, we have a homomorphism $\tsig:F\to \sO_X$ that realize 
$\bsig$ as a chain map.
\begin{lemm}\label{13}
Let $\bsig: \bbE\ldot\to \sO_X[-1]$ be a cosection of a symmetric obstruction theory $\phi:\bbE\udot\to \bbL_X$ on $X$, perfect of amplitude $[-2,0]$. There is a locally free symmetric resolution \eqref{1} of $\bbE\udot$ and a homomorphism $\tsig:F\to \sO_X$ of coherent sheaves such that $\bsig:\bbE\ldot\to \sO_X[-1]$ is represented by the chain map
\beq\label{14}\xymatrix{
\bbE\ldot\ar[d]_{\bsig} &B\ar[r]^d\ar[d] & F\ar[r]^{d^\vee}\ar[d]^{\tsig} & B^\vee\ar[d]\\
\sO_X[-1]&0\ar[r]&\sO_X\ar[r] &0 .
}\eeq
\end{lemm}
\begin{proof}
In the proof of Proposition \ref{10}, it is easy to see that we may further assume that
$$\bbE\udot \cong \bbE\ldot [2] \mapright{\bsig} \sO_X[1]$$
is represented  by a chain map
\beq\label{18}\xymatrix{
\bbE\udot\ar[d]_{\bsig\circ\theta}  & A^{-2}\ar[r]^{d_2} \ar[d] & A^{-1}\ar[r]^{d_1}\ar[d]^{\gamma} & A^0\ar[d]\\
\sO_X[1] & 0\ar[r] & \sO_X\ar[r] &0.
}\eeq
As the composition
$$A^{-2}\xrightarrow{d_2\oplus \theta_2^\vee} A^{-1}\oplus A_0 \xrightarrow{\gamma\oplus 0} \sO_X$$
is zero, $\gamma\oplus 0$ defines a homomorphism
$$\tsig:F\lra \sO_X$$
by \eqref{17} which fits into \eqref{14}. 
%
\end{proof}
We will see examples of cosections in \S\ref{S6} when $X$ is a quasi-projective moduli space of stable sheaves or perfect complexes on a Calabi-Yau 4-fold. 

\medskip

An orientation of the obstruction theory $\phi$ induces an orientation of $F$. 
\begin{prop}\cite[Proposition 4.2]{OhTh}
The choice of an orientation on $\bbE\udot$ 
induces an orientation of $F$, which makes $F$ an $SO(2n)$-bundle.\end{prop}

%
%

In \cite{OhTh}, Oh and Thomas define the virtual cycle for an oriented symmetric obstruction theory, perfect of amplitude $[-2,0]$ as follows. 
\begin{defi}\label{195} \cite[Definition 4.4]{OhTh} 
Let $\phi:\bbE\udot\to \bbL_X$ be a symmetric obstruction theory, perfect of amplitude $[-2,0]$, with an orientation. If $\rank \bbE\udot$ is odd, the virtual cycle $[X]\virt$ is defined to be zero.

Let $\rank \bbE\udot$ be even. The expected dimension of $X$ is $vd=\frac12 \rank \bbE\udot$. 
Pick a resolution \eqref{1} with $F$ an $SO(2n)$-bundle.
Let \beq\label{n83} C=F\times_{[F/B]}\fC_X\subset F\eeq
where $\fC_X$ is the intrinsic normal cone of $X$ embedded into the vector bundle stack $[F/B]$ by the 
perfect obstruction theory (cf. \cite{BeFa})
$$[F^\vee\to B^\vee]\lra [B\to F^\vee\to B^\vee]\cong\bbE\udot\mapright{\phi} \bbL_X$$
where the first arrow is the chain map $(0,\id_{F^\vee},\id_{B^\vee})$. 
Assuming that $C$ is isotropic in $F$ and 
letting $\tau $ denote the tautological section of $F|_C$ induced by the inclusion \eqref{n83}, the virtual cycle of $X$ is defined by 
\beq\label{n84} [X]\virt=\sqe(F|_C,\tau)[C]\in A_{vd}(X).\eeq
\end{defi}
Let $Y=C_\redd$ be the reduced stack of $C$. Then \eqref{n84} equals
\beq\label{i25} [X]\virt=\sqe(F|_Y,\tau)[C]\eeq
by applying Lemma \ref{102} (1) to the inclusion $Y\hookrightarrow C$ with $A_*(C)=A_*(Y)$.

Let $X$ be a quasi-projective moduli space of stable sheaves or complexes on a Calabi-Yau 4-fold. By \cite{CGJ}, there is an orientation for the symmetric obstruction theory $$\bbE\udot=\tau^{[-2,0]}\Big( Rp_*R\cH om(\cE,\cE)[3]\Big)$$ due to \cite[Theorem 4.1]{HuTh} and \cite{ToVa} where $\cE$ denotes the universal family and $p$ denotes the projection to $X$. 
By \cite[Proposition 4.3]{OhTh} and \cite{BBBJ,BBJ}, the cone $C$ is isotropic in $F$ and hence $\tau$ is an isotropic section of $F|_C$. By \cite[Proposition 4.5]{OhTh}, \eqref{n84} is independent of the choice of the resolution \eqref{1}. By \cite[\S4.3]{OhTh}, $[X]\virt$ is deformation invariant. 
Algebraically, DT4 invariants are defined as integrals of cohomology classes against $[X]\virt$. 

In the subsequent section, we will localize \eqref{n84} to the zero locus of a cosection $\sigma=h^1(\bsig)$ when $\phi$ admits an isotropic cosection $\bsig$.

\bigskip

\section{Cosection localized virtual cycles}\label{S5n}

In this section, we construct cosection localized virtual cycles for \DM stacks equipped with an oriented symmetric 3-term perfect obstruction theories. 

Throughout this section, we will make the following assumption.
\begin{assu}\label{n86} 
Let $X$ be a \DM stack equipped with a \emph{symmetric obstruction theory} $\phi:\bbE\udot\to \bbL_X$ with $\rank \bbE\udot$ even, perfect of amplitude $[-2,0]$ so that we have a resolution 
\beq\label{n87} \bbE\udot\cong [B\mapright{d} F\cong F^\vee\mapright{d^\vee} B^\vee].\eeq 
There is an \emph{orientation} $or:\sO_X\cong \det \bbE\ldot$ so that $F$ is an $SO(2n)$-bundle. 
Moreover, the intrinsic normal cone of $X$ is \emph{isotropic} in the sense that the induced cone \eqref{n83} is isotropic in $F$ for any resolution \eqref{n87}. 
We further assume that $X$ admits a cosection $\bsig:\bbE\ldot\to \sO_X[-1]$ which gives us a cosection \beq\label{n91}\tsig:F\lra \sO_X\eeq of the $SO(2n)$-bundle $F$ by Lemma \ref{13}. Let $vd=\frac12\rank \bbE\udot$. 
\end{assu}
Under Assumption \ref{n86}, 
we have the virtual cycle $[X]\virt$ defined by \eqref{n84} or \eqref{i25}. 
When $X$ is a quasi-projective moduli space of stable sheaves or complexes on a Calabi-Yau 4-fold, as we saw at the end of \S\ref{S5}, 
Assumption \ref{n86} is satisfied.

\medskip

\subsection{Localization by isotropic cosection}\label{S5.2}

A cosection $\bsig:\bbE\ldot\to \sO_X[-1]$ of the obstruction theory $\phi$ is called \emph{isotropic} if its square  
\beq\label{n94}
\sigma\ldot^2=\bsig\circ \theta\circ \sigma\ldot^\vee:\sO_X\lra \bbE\udot [-1]\mapright{\theta} \bbE\ldot [1] \lra \sO_X\eeq
is zero. If $\bsig$ is isotropic, then $\tsig:F\to \sO_X$ is isotropic, in the sense that $\tsig^\vee\cdot \tsig^\vee=\tsig(\tsig^\vee)=0$ by \eqref{14}. 

\begin{theo}\label{n88}
Under Assumption \ref{n86}, if there is an isotropic cosection $\bsig:\bbE\ldot\to \sO_X[-1]$, 
we have the localized virtual cycle
\beq\label{123.1} [X]\virt_\loc\in A_{vd}(X(\sigma))\eeq
defined by \eqref{n89} 
where $X(\sigma)$ is the zero locus 
of the cosection 
$$\sigma=h^1(\bsig):Ob_X\lra \sO_X.$$ Letting $\imath:X(\sigma)\to X$ denote the inclusion map, we have \beq\label{n90} \imath_*[X]\virt_\loc=[X]\virt\eeq
and $[X]\virt_\loc$ is deformation invariant in the sense of Lemma \ref{125} below.
\end{theo}

\begin{proof}
By \eqref{14}, $\tsig\circ d=0$ and by the cone reduction lemma \cite{KLc}, the reduced cone
$Y=C_\redd$ of \eqref{n83} lies in the kernel $F(\tsig)$ of $\tsig$.
Hence the tautological section $\tau$ of $F|_Y$ induced by the inclusion $Y\subset F$ satisfies 
$\tsig|_Y\circ \tau=0$ and hence 
$\tsig|^\vee_Y\cdot \tau=0.$ Note that $X(\sigma)\supset X(\tsig)$. 

We claim that the isotropic sections $\tsig|^\vee_Y$ and $\tau$ are independent away from 
$$Z=Y\times_XX(\sigma)\supset Y\times_XX(\tsig)=(\tsig|_Y^\vee)^{-1}(0).$$
Since $\tau^{-1}(0)=X_\redd$,  by Theorem \ref{n74}, 
the localized square root Euler class $$\sqe(F|_Y,\tau;\tsig|^\vee_Y):A_*(Y)\lra A_*(X(\sigma))$$ gives us the localized virtual cycle 
\beq\label{n89} [X]\virt_\loc:=\sqe(F|_Y,\tau;\tsig|^\vee_Y)[C]\in A_{vd}(X(\sigma)).\eeq
The equality \eqref{n90} follows immediately from \eqref{n75} and \eqref{i25}. 

Let us prove the independence of $\tsig^\vee|_Y$ and $\tau$. 
The notion of independence (Definition \ref{n73}) is local and we may assume that $X$ is an affine scheme and $\sigma$ is surjective. 
 Since the composition 
$$\mathrm{ker}(d\dual:F \to B\dual) \lra Ob_{X} \mapright{\sigma} \sO_{X}$$ 
is surjective, there is a section $a \in H^0(X,F)$ such that $\tsig (a) = 1$ and $d\dual(a) = 0$. Replacing $a$ by $a-\frac{1}{2}a^2\tsig\dual$, we may assume that $a$ is isotropic. Hence, we have an orthogonal decomposition
$$F = \langle a , \tsig\dual \rangle \oplus \langle a,\tsig\dual \rangle^{\perp},$$
where $\langle a , \tsig\dual \rangle$ is the image of $(a,\tsig\dual):\sO_X^2 \xrightarrow{} F$.
Since both 
$$a\dual, \tsig : F \lra \sO_X$$
are cosections of the perfect obstruction theory $[F\to B\dual] \to \bbL_X$, the cone reduction lemma \cite[Proposition 4.3]{KLc} tells us that 
\beq\label{n92}
Y=C_{\redd} \subseteq \mathrm{ker}(a\dual)\cap \mathrm{ker}(\tsig) = \langle a,\tsig\dual\rangle^{\perp}
\cong \langle\tsig\dual\rangle^{\perp}/\langle\tsig\dual\rangle.\eeq
Thus $F|_Y=F'\oplus F''$, $\tau\in H^0(F')$ and $\tsig^\vee|_Y\in H^0(F'')$ where
$$F'=\langle a,\tsig\dual\rangle^{\perp}|_Y,\and F''=\langle a,\tsig\dual\rangle|_Y.$$

By Lemma \ref{124} below, \eqref{n89} is independent of choices and by Lemma \ref{125}, $[X]\virt_\loc$ is deformation invariant.  
This completes our proof. 
\end{proof}
By \eqref{n92}, we have the closed immersion  
\beq\label{i39} Y=C_\redd\hookrightarrow \langle\tsig\dual\rangle^{\perp}/\langle\tsig\dual\rangle\eeq
which is an orthogonal version of the cone reduction lemma \cite[Proposition 4.3]{KLc}. We will use this  for a construction of reduced virtual cycles in \S\ref{S5.3}.

As a corollary, the virtual cycle vanishes if the cosection is surjective.

\begin{coro}\label{126}
If ${\sigma}:Ob_X \to \sO_X$ is surjective, then $[X]\virt=0$.
\end{coro}

We now prove that the localized virtual cycle is well defined.

\begin{lemm}\label{124}
The localized virtual cycle \eqref{123.1} is independent of the choices of the resolution \eqref{n87} and the lifting \eqref{n91}.
\end{lemm}

\begin{proof}
We first prove that the localized virtual cycle \eqref{123.1} is independent of the choice of the lifting \eqref{n91} by a deformation argument. Let $\tsig,\tsig':F \to \sO_X$ be two liftings of $\sigma$. Consider the family of cosections parametrized by $\bbA^1$
$$\Sigma= (1-t)\tsig + t \tsig' : F|_{X \times \bbA^1} \to \sO_{X \times \bbA^1},$$
whose fiber over $0$ is $\tsig$ and the fiber over $1$ is $\tsig'$. 
Note that the two sections
$$\tau|_{Y\times \bbA^1},\Sigma|_{Y\times \bbA^1}\dual \in H^0(Y \times \bbA^1, F|_{Y\times \bbA^1})$$ 
are mutually orthogonal isotropic sections, which are independent away from $(Y\times_X X(\sigma)) \times \bbA^1$. Indeed, this can be shown by considering its fibers over $t \in \bbA^1$. 
Let $\imath_t:X(\sigma)\times\{t\} \hookrightarrow X(\sigma)\times \bbA^1$ be the inclusion map. Then we have
\begin{equation}\label{124.1}
\imath_t^*\sqrt{e}(F|_{Y\times \bbA^1},\tau;\Sigma\dual)[C\times \bbA^1] = \sqrt{e}(F|_{Y},\tau; \Sigma_t\dual)[C] \in A_* (X({\sigma}))
\end{equation}
because the localized square root Euler class \eqref{n76} is bivariant by Theorem \ref{n74}. 
Since $$\sqrt{e}(F|_{Y\times \bbA^1},\tau;\Sigma\dual)[C\times \bbA^1]=p^*\xi\in A_*(X(\sigma)\times\bbA^1)\cong p^*A_{*-1}(X(\sigma))$$
for some $\xi\in A_*(X(\sigma))$ where $p:X(\sigma)\times \bbA^1\to X(\sigma)$ denotes the projection, $$\imath_t^*\sqrt{e}(F|_{Y\times \bbA^1},\tau;\Sigma\dual)[C\times \bbA^1]=\imath_t^*p^*\xi=(p\circ \imath_t)^*\xi=\xi$$ for all $t$. 
Hence, \eqref{124.1} is independent of $t \in \bbA^1$ and the two localized virtual cycles given by $\tsig$ and $\tsig'$ are equal.

In \cite[\S4.3]{OhTh}, it was proved that $[X]\virt$ is independent of the choice of \eqref{n87}. The same arguments prove that  \eqref{123.1} is independent of the choice of \eqref{n87} with 
the bivariance of $\sqe(F,s;t)$ in Theorem \ref{n74} and Corollary \ref{117}. 
\end{proof}

The localized virtual cycle $[X]\virtloc$ is deformation invariant.

\begin{lemm}\label{125}
Let $X \to S$ be a morphism from a \DM stack to a smooth scheme $S$. Let $\phi:\bbE\udot \to \bbL_{X/S}$ be a relative symmetric obstruction theory of amplitude $[-2,0]$ equipped with an orientation. Assume that the lift $C=\fC_{X/S} \times_{[F/B]}F \subseteq F$ of the relative intrinsic normal cone $\fC_{X/S}$ is isotropic for any symmetric resolution $[B\to F \to B\dual] \cong \bbE\udot$. Suppose we have an isotropic cosection $\bsig:\bbE\ldot \to \sO_X[-1]$ such that the composition
\begin{equation}\label{125.2}
\sO_X \mapright{\sigma\ldot\dual} \bbE\udot[-1] \mapright{\phi} \bbL_{X/S}[-1] \mapright{\delta} \bbL_S|_X=\Omega_S|_X,    
\end{equation}
is zero. Let $Ob_{X/S}=h^{1}(\bbE\ldot)$ and
$$Ob_X=\ker(Ob_{X/S}\xrightarrow{h^0(\delta\circ\phi\circ \theta^{-1})} \Omega_S|_X)/\im (\bbT_S|_X\xrightarrow{h^0(\delta\circ\phi)^\vee} Ob_{X/S}).$$
As \eqref{125.2} is zero, $h^1(\sigma\ldot):Ob_{X/S} \to \sO_X$ induces a cosection $\sigma:Ob_X \to \sO_X$. Let $X({\sigma})$ be the zero locus of ${\sigma}$. Then there exists a localized virtual cycle  $[X/S]\virtloc \in A_* (X({\sigma}))$ such that for any $p \in S$,
\begin{equation}\label{125.1}
\imath_p^![X/S]\virtloc=[X_p]\virtloc \in A_* (X_p({\sigma}))    
\end{equation}
where $\imath_p:\{p\} \to S$ denotes the inclusion.
\end{lemm}

\begin{proof}
Choose a symmetric resolution $\bbE\udot\cong [B\xrightarrow{d}F\xrightarrow{d\dual}B\dual ]$ and a homomorphism $\tsig:F \to \sO_X$ that fit into \eqref{14}.
We may also assume that the composition $$\bbE\udot \mapright{\phi} \bbL_{X/S} \mapright{\delta} \bbL_S|_X[1]=\Omega_S|_X[1]$$ is a morphism of complexes. 
Let $a: F \to \Omega_S|_X$ be the induced map such that $a\circ d=0$. Then there exists an absolute perfect obstruction theory
\begin{equation}\label{125.3}
\psi:{\bbF}\udot=[F \xrightarrow{(d\dual,a)}B\dual\oplus \Omega_S|_X] \lra \bbL_X    
\end{equation}
which fits into the morphism of distinguished triangles
\begin{equation}
\xymatrix{
\Omega_S|_X \ar[r] \ar@{=}[d] & {\bbF}\udot \ar[r] \ar[d]^{\psi} & [F\to B\dual] \ar[r] \ar[d] & \Omega_S|_X[1] \ar@{=}[d]\\
\Omega_S|_X \ar[r] & \bbL_X \ar[r] & \bbL_{X/S} \ar[r] & \Omega_S|_X[1]}    
\end{equation}
over $X$. Hence we obtain a fiber diagram
$$\xymatrix{
\fC_{X/S} \ar[r] \ar[d] & [F/B]\ar[d]\\
\fC_X \ar[r] &[F/B\oplus \bbT_S] 
}$$
of cone stacks over $X$. Let $C=\fC_{X/S}\times_{[F/B]}F$ and let $\tau \in H^0(C,F)$ be the tautological section.

Since \eqref{125.2} is zero by assumption, $\tsig:F \to \sO_X$ is a cosection for the perfect obstruction theory \eqref{125.3}. Hence the cone reduction lemma \cite{KLc} implies that the composition $\tsig \circ \tau $ vanishes on the support $Y=C_\redd$ of $C$. By the arguments in the proof of Theorem \ref{n88}, the sections $\tau$ and $\tsig\dual$ are independent away from $Y \times_X X(\sigma)$. Hence we obtain a localized virtual cycle 
\begin{equation}\label{125.6}
[X/S]\virtloc=\sqrt{e}(F|_Y,\tau;\tsig|_Y\dual)[C] \in A_*(X(\sigma)). \end{equation}
Consider the fiber diagram
$$\xymatrix{
C_p \ar[r] \ar[d] & C|_{X_p} \ar[r] \ar[d] & F\ar[d]\\
\fC_{X_p} \ar[r] & \fC_{X/S}|_{X_p} \ar[r] & [F/B].}$$
By Vistoli's rational equivalence $\imath_p^![C]=[C_p] \in A_*(C|_{X_p})$, applying $\imath_p^!$ to \eqref{125.6}, we get 
$$\imath_p^![X/S]\virtloc = [X_p]\virtloc \in A_* (X_p({\sigma}))$$
by the bivariance of $\sqrt{e}(F|_Y,\tau;\tsig|_Y\dual)$.
\end{proof}


%

\begin{rema}\label{129}
From the proof of Theorem \ref{n88}, we see that it suffices to 
assume that the restriction $\bsig|_{X_{\redd}}$ is isotropic in order to define the localized virtual cycle \eqref{123.1}.
\end{rema}

\medskip

\subsection{Reduced virtual cycles}\label{S5.3} 
By Corollary \ref{126}, when there is a surjective isotropic cosection of the obstruction sheaf $Ob_X$, the virtual cycle $[X]\virt$ vanishes under Assumption \ref{n86}. As in the case of Gromov-Witten theory for K3 surfaces \cite{KoTh}, we can construct a reduced virtual cycle as follows.

\begin{defi}\label{131}
Under Assumption \ref{n86}, suppose that $\bsig$ is isotropic and $\sigma:h^1(\bbE\ldot)=Ob_X\to \sO_X$ is surjective. 
Choose a symmetric resolution \eqref{n87} 
and a lift \eqref{n91} of $\sigma$ which has to be surjective as $\sigma$ is surjective.  
Let $Y=C_\redd$  be the reduced stack of $C:=\fC_X \times_{[F/B]}F\subset F$ which is embedded into the $SO(2n-2)$-bundle 
\beq\label{n93} F_{\tsig}:=\langle\tsig\dual\rangle^{\perp} /\langle\tsig\dual\rangle\eeq
as an isotropic cone by \eqref{i39}. 
We define the \emph{reduced virtual cycle} of $X$ by 
\begin{equation}\label{131.1}
[X]\virt_\redd:=\sqrt{e}(F_{\tsig}|_Y,\tau)[C] \in A_{vd +1}(X).    
\end{equation}
where $\tau \in H^0(Y,F_{\tsig}|_Y)$ denotes the tautological section.
\end{defi}

\begin{rema}\label{132}
A reduced obstruction theory $[B \to F_{\tsig} \to B\dual] \to \bbL_X$ may not exist in general. In \cite[A.3]{MPT}, it was proved that a perfect obstruction theory $\psi :\bbF \to \bbL_X$ in the sense of Behrend-Fantechi with a surjective cosection $h^1(\bbF\dual) \to \sO_X$ may not have a reduced obstruction theory. This implies that the symmetric obstruction theory $\phi=(\psi,0):\bbE=\bbF\oplus \bbF\dual[2] \to \bbL_X$ may not have a reduced obstruction theory.
\end{rema}

\begin{lemm}\label{133}
The reduced virtual cycle $[X]\virt_{\redd}$ in \eqref{131.1} is independent of the choices of the symmetric resolution \eqref{n87} and the lift \eqref{n91}.
\end{lemm}
\begin{proof}
Using Corollary \ref{105} and Lemma \ref{102} (3), Lemma \ref{133} follows from the arguments in the proof of Lemma \ref{124}.
\end{proof}

The reduced virtual cycle $[X]\virt_{\redd}$ in \eqref{131.1} is deformation invariant.

\begin{prop}\label{134}
Under the assumptions of Lemma \ref{125}, suppose that ${\sigma}:Ob_X \to \sO_X$ is surjective. Then there exists a cycle  $[X]\virt_{\redd} \in A_*(X)$ such that for any $p \in S$,
$$\imath_p^![X]\virt_{\redd}=[X_p]\virt_{\redd} \in A_* (X_p)$$
where $\imath_p:\{p\} \to S$ denotes the inclusion.
\end{prop}
\begin{proof}
Using the notation of Lemma \ref{125},  
the arguments for \eqref{n92} prove that  we have a closed immersion $Y \hookrightarrow F_{\tsig}$ where $F_{\tsig}$ is defined by \eqref{n93}. Let
$$[X]\virt_\redd:=\sqrt{e}(F_{\tsig}|_Y,\tau)[C] \in A_*(X)$$
where $\tau$ is the tautological section. Then Lemma \ref{102} (3) with Vistoli's rational equivalence completes the proof as in the proof of Lemma \ref{125}.
\end{proof}

\begin{rema}\label{135}
(1) As in Remark \ref{129}, we only need to assume that $\bsig|_{X_{\redd}}$ is isotropic, to define the reduced virtual cycle $[X]\virt_{\redd}$ 
in Definition \ref{131}.

(2) It is also straightforward to define a reduced virtual cycle when there are multiple cosections.
In Definition \ref{131}, let
$$\sigma\ldot^1,\sigma\ldot^2,\cdots,\sigma\ldot^k: \bbE\ldot \lra \sO_X[-1]$$
be mutually orthogonal isotropic cosections 
such that
$$(\sigma^1,\cdots,\sigma^k)=h^1(\sigma\ldot^1,\sigma\ldot^2,\cdots,\sigma\ldot^k):Ob_X \lra \sO_X^{\oplus k}$$
is surjective. Then we may choose a symmetric resolution $[B \to F \to B\dual]\cong \bbE\udot$ and lifts $\tsig^1,\cdots,\tsig^k:F \to \sO_X$ of $\sigma^1,\cdots,\sigma^k$ which are isotropic and mutually orthogonal. Hence the span $\langle (\tsig^1)\dual,\cdots,(\tsig^k)\dual\rangle$ is an isotropic subbundle of $F$. Let
$$F_{\tsig^1,\cdots,\tsig^k}=\langle (\tsig^1)\dual,\cdots,(\tsig^k)\dual\rangle^{\perp}/\langle (\tsig^1)\dual,\cdots,(\tsig^k)\dual\rangle$$
be the reduced $SO(2n-2k)$-bundle. By \eqref{n92}, the support $Y=C_{\redd}$ of the cone $C=\fC_X \times_{[F/B]}F$ is contained in $F_{\tsig^1,\cdots,\tsig^k}$. We can define the reduced virtual cycle as
$$[X]\virt_{\redd}=\sqrt{e}(F_{\tsig_1,\cdots,\tsig_k}|_Y,\tau)[C] \in A_{vd+k}(X)$$
where $\tau$ is the tautological section.
\end{rema}

\medskip

\subsection{Vanishing for non-isotropic cosections}\label{S5.4}

Up to this point in this section, we considered only isotropic cosections, i.e. \eqref{n94} is zero.  
Now, we prove a vanishing result of virtual cycles for non-isotropic cosections.


\begin{theo}\label{143}
Under Assumption \ref{n86}, if $\sigma\ldot^2$ in \eqref{n94} is a nonzero constant, 
then the virtual cycle vanishes:
\begin{equation}
[X]\virt = 0 \ \ \in A_{vd} (X).    
\end{equation}
\end{theo}

We first prove a vanishing result for the ordinary square root Euler class $\sqrt{e}(F)$ of Edidin-Graham \cite{EdGr}.

\begin{lemm}\label{141}
Let $F$ be an $SO(2n)$-bundle over a \DM stack $Y$. If there is a section $t$ of $F$ such that $t^2 \in \CC^*$ is a nonzero constant, then
\begin{equation}
\sqrt{e}(F)=0.    
\end{equation}
\end{lemm}
\begin{proof}
As we saw in the proof of Theorem \ref{n55}, by the bivariance of $\sqe(F)$, we may replace $Y$ by the isotropic flag variety 
$Q$ in \S\ref{S2.1} and assume that $F$ admits a positive maximal isotropic subbundle $V$ and $\sqe(F)=e(V)$. 
By further replacing $Y$ by an affine bundle (see the discussion after Lemma \ref{n56}), we may assume $F=V\oplus V^\vee$. 
Hence $t=(t_2,t_1)$ with $t_2\in H^0(V)$ and $t_1\in H^0(V^\vee)$.  As $t^2=2t_1t_2\in \CC^*$, we find that $t_1$ and $t_2$ are both nowhere vanishing. 
As $V$ admits a nowhere vanishing section, $e(V)=0$ and hence 
$\sqe(F)=e(V)=0$ 
as desired. 
\end{proof}

Lemma \ref{141} implies a vanishing result for the localized square root Euler class $\sqrt{e}(F,s)$.

\begin{lemm}\label{142}
Let $F$ be an $SO(2n)$-bundle over a \DM stack $Y$, and $s$ be an isotropic section. Assume that there is a section $t$ of $F$ such that $s\cdot t=0$ and $t^2 \in \CC^*$. Then
\begin{equation}
\sqrt{e}(F,s)=0.    
\end{equation}
\end{lemm}
\begin{proof}
Using the notation in Definition \ref{n39}, 
$t$ induces a section $\tit$ of $\tF$ with $\tit^2\in \CC^*$. 
By Lemma \ref{141},
$$\sqrt{e}(F)=0 \and \sqrt{e}(\tF)=0.$$ By \eqref{n41}, we have the vanishing $\sqe(F,s)=0$.
\end{proof}

\begin{proof}[Proof of Theorem \ref{143}]
Choose a resolution \eqref{n87} 
and a lift $\tsig:F \to \sO_X$ of the cosection $\sigma=h^1(\bsig)$ such that $\tsig \circ d=0$ by Lemma \ref{13}. 
Let $t=\tsig^\vee$. Then by \eqref{14}, \eqref{n94} being a nonzero constant implies that $t^2 \in \CC^*$ is a nonzero constant. 
Let $C=\fC_X \times_{[F/B]}F$ and $Y=C_\redd$.  
By the cone reduction lemma \cite{KLc}, $\tsig(Y)=0$ and hence the composition
$$\tau\cdot t|_Y=\tsig|_Y\circ \tau:\sO_Y \mapright{\tau} F|_Y \mapright{\tsig} \sO_Y$$ is zero  
where $\tau$ is the tautological section. By \eqref{i25} and Lemma \ref{142},
$$[X]\virt=\sqrt{e}(F|_Y,\tau)[C]=0 \in A_*(X)$$ as desired.
\end{proof}

\begin{rema}
(1) As in Remark \ref{129}, we only need to assume $\sigma\ldot^2|_{X_{\redd}} \in \CC^*$ to obtain the vanishing result in Theorem \ref{143}.

(2) In the situation of Theorem \ref{143}, if there is an additional cosection $\lambda\ldot:\bbE\ldot \to \sO_X[-1]$ such that $\lambda\ldot^2\in \CC^*$ and  $\sigma\ldot\cdot\lambda\ldot=0$, then 
$$(\lambda\ldot^2)\sigma\ldot+\sqrt{-1}(\sigma\ldot^2)\lambda\ldot$$ is an isotropic cosection such that the induced map $Ob_X \to \sO_X$ is surjective. Hence we can define a reduced virtual cycle as in Definition \ref{131}.
\end{rema}


\bigskip

\section{Cosections for DT4 moduli spaces}\label{S6}

We provide examples of cosections on moduli spaces of sheaves or complexes on Calabi-Yau 4-folds.

\subsection{Cosections for DT4 moduli without fixing the determinant}\label{S6.1}

Let $W$ be a Calabi-Yau 4-fold. Let $X$ be a component of the moduli space of simple perfect complexes on $W$ with fixed topological type  $c$ which is an algebraic space by \cite{Lieblich,Inaba}. By \cite{HuTh}, there is a 3-term symmetric obstruction theory
\beq\label{n96}\phi:\bbE\udot \lra \bbL_X, \qquad \bbE\ldot=\tau^{[0,2]}Rp_*R\hom_{X \times W}(\cE,\cE)[1]\eeq
where $\cE$ is the universal complex
and $p:X\times W \to X$ is the projection while $\tau^{[0,2]}=\tau^{\ge 0}\tau^{\le 2}$ denotes the truncation. 
We will define cosections
$$\sigma\ldot:\bbE\ldot \lra \sO_X[-1],$$
associated to $(2,0)$-forms, $(3,1)$-forms, and $(0,2)$-forms on $W$. Let $$At(\cE) : \cE \lra \cE \otimes \Omega_{W}[1]$$ denote the relative Atiyah class.

We first consider the cosections defined by holomorphic 2-forms on $W$ due to Cao-Maulik-Toda.
\begin{exam}\cite{CMT,CMT2}\label{151}
Let $\theta \in H^0(W,\Omega^2_W)$ be a holomorphic 2-form on a Calabi-Yau 4-fold $W$. Then $X$ admits a cosection
\begin{equation}
\sigma^{\theta}\ldot:\bbE\ldot \lra \sO_X[-1]    
\end{equation}
defined by 
$$\begin{aligned}
R&p_*R\hom(\cE,\cE)[1] \xrightarrow{\circ At^2(\cE)} Rp_*R\hom(\cE,\cE \otimes \Omega^2_W)[3]\\
& \xrightarrow{\tr}Rp_*\Omega^2_W[3] \xrightarrow{\wedge \theta}Rp_*\Omega^4_W[3] \lra \sO_X[-1]
\end{aligned}$$
where
 the last arrow is the relative Serre duality.
\end{exam}

Moreover, any $(3,1)$-form on $W$ defines a cosection on $X$. This was implicitly considered in the proof of \cite[Proposition 2.9 (2)]{CMT}.

\begin{exam}\label{152}
Let $\delta \in H^1(W,\Omega^3_W)$. 
Define a cosection
\begin{equation}\label{152.1}
\sigma^{\delta}\ldot:\bbE\ldot \lra \sO_X[-1]    
\end{equation}
by the composition
$$\begin{aligned}
R&p_*R\hom(\cE,\cE)[1] \xrightarrow{\circ At(\cE)} Rp_*R\hom(\cE,\cE \otimes \Omega_W)[2]\\ &\xrightarrow{\tr}Rp_*\Omega_W[2] \xrightarrow{\wedge\delta}Rp_*\Omega^4_W[3] \lra\sO_X[-1]
\end{aligned}$$
where  the last arrow is the relative Serre duality.
\end{exam}

\begin{exam}\label{153}
Let $\gamma \in H^2(W,\sO_W)$ be a $(0,2)$-form on $W$. Define a cosection
\begin{equation}\label{153.1}
\sigma^\gamma\ldot:\bbE\ldot \lra  \sO_X[-1]    
\end{equation}
as the composition
$$Rp_*R\hom(\cE,\cE)[1] \xrightarrow{\tr}Rp_*\sO[1] \xrightarrow{\gamma}Rp_*\sO[3] \mapright{\omega} Rp_*\Omega^4_W[3]\lra \sO_X[-1]$$
where we used a fixed Calabi-Yau 4-form $\omega$ and Serre duality.
\end{exam}

We thank Richard Thomas for suggesting the cosections \eqref{152.1} and \eqref{153.1}.

\medskip

For any two cosections $\sigma\ldot,\sigma'\ldot:\bbE\ldot \to \sO_X[-1]$, the product $\sigma\ldot\cdot\sigma'\ldot=\sigma'\ldot\cdot\sigma\ldot$ is defined as the composition
$$\sO_X \xrightarrow{\sigma\ldot\dual} \bbE\udot[-1] \cong \bbE\ldot[1] \xrightarrow{\sigma'\ldot} \sO_X.$$

\begin{lemm}\label{154}
Let $W$ be a Calabi-Yau 4-fold and $\omega \in H^0(W,\Omega^4_W)$ be nowhere vanishing. 
By $\omega$, we have an isomorphism 
$$\iota:H^1(W,\Omega^3_W)\mapright{\omega^{-1}} H^1(W,\bbT_W),\quad \delta\mapsto \iota_\delta.$$
Let $c_i$ denote the $i$-th Chern class of the perfect complexes $[E]$ parameterized by $X$ and $c_0$ denote their rank.  
\begin{enumerate}

\item \cite[Prosition 2.9 (2)]{CMT} For $\theta\in H^0(W,\Omega^2_W)$, $\delta \in H^1(W,\Omega^3_W)$, we have
$$\sigma^\theta\ldot\cdot\sigma^\delta\ldot|_{X_\redd} = 2\int_W \iota_\delta\theta \cup c_3.$$
\item For $\delta_1,\delta_2 \in H^1(W,\Omega^3_W)$, we have
$$\sigma^{\delta_1}\ldot\cdot\sigma^{\delta_2}\ldot|_{X_\redd}=\int_W(\iota_{\delta_1}\iota_{\delta_2}c_2)\cup \omega.$$
\item For $\gamma_1,\gamma_2 \in H^2(W,\sO_W)$, we have
$$\sigma^{\gamma_1}\ldot\cdot\sigma^{\gamma_2}\ldot|_{X_\redd}=c_0\int_W (\gamma_1 \cup \gamma_2 \cup \omega).$$
\end{enumerate}
\end{lemm}

\begin{proof}
(1) It was proved in \cite[Proposition 2.9(2)]{CMT} using \cite[Proposition 4.2]{BuFl}.

(2) For any $[E] \in X$, the extension 
$$({\sigma^\delta\ldot})\dual([E]) \in Ext_W^2(E,E)$$
is given by the composition
$$({\sigma^\delta\ldot})\dual([E]):E \xrightarrow{At(E)} E \otimes \Omega_W[1] \xrightarrow{1\otimes\iota_\delta}  E[2].$$
Then we have 
\begin{align*}
{\sigma^{\delta_1}\ldot}\cdot\sigma\ldot^{\delta_2}([E]) 
& = \tr((1\otimes\iota_{\delta_1}) \circ At(E) \circ (1\otimes \iota_{\delta_2}) \circ At(E))\\
& = \iota_{\delta_1} \tr(At(E) \circ (1\otimes\iota_{\delta_2}) \circ At(E))\\
& = \iota_{\delta_1} \tr((1\otimes\iota_{\delta_2}) \circ At(E)) \circ At(E) )\\
& = \iota_{\delta_1} \iota_{\delta_2}ch_2(E)
\end{align*}
by \cite[Proposition 4.2]{BuFl}.

(3) For any $[E]\in X$, we have
$$\sigma\ldot^\gamma([E])=(1\otimes \gamma): E \lra E[2].$$
Hence the equation $\tr \circ 1_E = \rank(E)=c_0$ proves (3).
\end{proof}

When the moduli space $X$ is quasi-projective, there is a symmetric resolution $[B\to F \to B\dual] \cong \bbE\udot$ so that we can define a virtual cycle
\beq\label{197}[X]\virt \in A_{vd}(X)\eeq
by Definition \ref{195}, where $vd=\frac12\rank \bbE\udot$. 

\begin{coro}\label{196}
If $W$ is a hyperk\"ahler 4-fold, $X$ is a projective moduli space and $c_3 \neq 0$, then $[X]\virt=0$. Moreover, if $c_2=0$ and $c_3\neq 0$, then there is a reduced virtual cycle 
\beq\label{196.1}[X]\virt_{\redd} \in A_{vd+1}(X).\eeq
\end{coro}

\begin{proof}
Let $\theta \in H^0(W, \Omega^2_W)$ be a holomorphic symplectic form. By Serre duality and non-degeneracy of $\theta$, there exists $\delta \in H^1(W,\Omega^3_W)$ such that
$$\sigma\ldot^{\theta}\cdot\sigma\ldot^{\delta}|_{X_\redd} =2 \int_W \iota_{\delta}\theta \cup c_3=1$$
because $c_3 \neq 0$. This implies that ${\sigma^{\delta}}:Ob_X \to \sO_X$ is surjective. If the cosection $\sigma\ldot^\delta$ is isotropic, then the virtual cycle vanishes by Corollary \ref{126}. On the other hand, when the square $(\sigma\ldot^\delta)^2$ of the cosection $\sigma\ldot^\delta$ is nonzero, then we deduce $[X]\virt=0$ by Theorem \ref{143} because the global function $(\sigma\ldot^\delta)^2$ of the projective scheme $X_\redd$ is constant.

If $c_2=0$, then the cosection $\sigma\ldot^\delta$ is isotropic by Lemma \ref{154} (2). Hence we can define a reduced virtual cycle \eqref{196.1} by Definition \ref{131}.
\end{proof}

\begin{coro}\label{199}
If $c_2=0$, any $(3,1)$-form $\delta \in H^1(W,\Omega^3_W)$ defines a localized virtual cycle 
$$[X]\virtloc \in A_{vd}(X({\sigma^\delta}))$$
where $X(\sigma^\delta)$ is the zero locus of the cosection $\sigma^\delta=h^1(\sigma^\delta\ldot)$. More generally, if the bilinear map 
\begin{equation}\label{199.1}
H^1(W,\bbT_W^1)\otimes H^1(W,\bbT_W^1) \lra \CC , \quad \delta_1\otimes\delta_2\mapsto \int_W(\iota_{\delta_1}\iota_{\delta_2}c_2)\cup \omega
\end{equation}
is zero, then any $(3,1)$-form defines a localized virtual cycle. 
If the rank $k$ of \eqref{199.1} is positive, then $[X]\virt=0$. Furthermore, if the rank $k$ is a positive even number, then there exists a reduced virtual cycle $[X]\virt_\redd \in A_{vd+k/2}(X).$
\end{coro}

\begin{proof}
If the bilinear map \eqref{199.1} is zero, then the cosection $\sigma\ldot^\delta:\bbE\ldot \to \sO_X[-1]$ is isotropic by Lemma \ref{154} (2). Hence we have the localized virtual cycle by Theorem \ref{n88}.

If the rank $k$ of \eqref{199.1} is positive, then there exist $\delta_1,\delta_2 \in H^1(W,\bbT_W)$ such that $\delta_1\cdot\delta_2=1$. This means that 
$${\sigma^{\delta_1}}:Ob_X\lra \sO_X$$
induced by the cosection $\sigma\ldot^{\delta_1}:\bbE\ldot \to \sO_X[-1]$ is surjective. Hence the virtual cycle vanishes by Corollary \ref{126} and Theorem \ref{143}.

Now assume that the rank $k$ of \eqref{199.1} is a positive even number. Let $\Gamma$ be a $k$-dimensional subspace of $H^1(W,\bbT_W)$ such that the restriction of the symmetric bilinear map \eqref{199.1} is nondegenerate. Let $\Lambda$ be a maximal isotropic subspace of $\Gamma$ and $\delta_1,\cdots,\delta_l \in \Lambda$ be a basis for $l=k/2$. Then
$$\sigma\ldot^{\delta_1}, \cdots,\sigma\ldot^{\delta_l}:\bbE\ldot \lra \sO_X[-1]$$
are mutually orthogonal isotropic cosections. Hence there exists a reduced virtual cycle $[X]\virt_{\redd}\in A_{vd+l} (X)$ by Remark \ref{135} (2).
\end{proof}

\begin{coro}\label{198}
If $c_0=0$, then any $(0,2)$-form $\gamma \in H^2(W,\sO_W)$ defines a localized virtual cycle
$$[X]\virtloc \in A_{vd}(X({\sigma^\gamma})).$$
If $c_0>0$ and $h^{0,2}$ is a positive even number, then $[X]\virt=0$ and there exists a reduced virtual cycle 
$$[X]\virt_{\redd} \in A_{vd+h^{0,2}/2} (X).$$
\end{coro}

\begin{proof}
If $c_0=0$, then the cosection $\sigma\ldot^\gamma$ is isotropic by Lemma \ref{154} (3). Hence we have a localized virtual cycle by Theorem \ref{n88}. When $c_0>0$, the bilinear map
\beq\label{198.2}H^2(W,\sO_W) \otimes H^2(W,\sO_W) \lra H^4(W,\sO_W) \lra \CC,\quad \alpha\otimes\beta\mapsto c_0\int_W\alpha\cup\beta\cup\omega\eeq
is a perfect pairing. Hence we can define the reduced virtual cycle as in Corollary \ref{199} using Lemma \ref{154} (3).
\end{proof}


\begin{exam}\label{162}
Let $X=M_{\beta,1}$ be the moduli space of 1-dimensional stable sheaves on $W$, considered in \cite{CMT} to define the genus zero Gopakumar-Vafa invariants. Note that $c=(0,0,0,\beta,1)$ and $\beta \neq 0$. By Corollary \ref{196}, Corollary \ref{199}, and Corollary \ref{198}, we have the following:
\begin{enumerate}
\item If $W$ is a hyperk{\"a}hler 4-fold, then $[M_{\beta,1}]\virt=0$, and there is a reduced virtual cycle 
$$[M_{\beta,1}]\virt_{\redd} \in A_2(M_{\beta,1}).$$
\item Any $(3,1)$-form $\delta \in H^1(W,\Omega^3_W)$ (resp. $(0,2)$-form $\gamma \in H^2(W,\sO_W)$) gives us a localized virtual cycle
$$[M_{\beta,1}]\virtloc \in A_1(M_{\beta,1}(\sigma^\delta)), \quad (\text{resp. } [M_{\beta,1}]\virtloc \in A_1(M_{\beta,1}(\sigma^\gamma))).$$
\end{enumerate}
Moreover, (1) holds if $W$ has a holomorphic 2-form $\theta$ such that $PD(\beta) \in H^1(W,\Omega^1_W)$ is contained in the image of $\theta:H^1(W,\bbT_W) \to \Omega^1(W,\Omega_W)$.
\end{exam}

\subsection{Moduli spaces with fixed determinant} \label{S6.2}

Let $X_L$ be a quasi-projective component of the moduli stack of simple perfect complexes on $W$ with fixed determinant $L$ and Chern character $c$. We assume that $c_0=\rank E>0$ for $[E]\in X_L$. Then $X_L$ is a \DM stack which has a derived enhancement \cite{STV} and a (-2)-shifted symplectic structure \cite{PTVV}. Hence  $X_L$ has a symmetric obstruction theory
$$\phi:\bbE\udot \lra \bbL_{X_L}, \qquad \bbE\ldot=Rp_*R\hom(\cE,\cE)_0[1]$$
of amplitude $[-2,0]$ where the subscript $0$ denotes the traceless part. 
When $X$ is projective,  we have a virtual cycle
$$[X_L]\virt \in A_{vd} (X_L)$$
where $vd=\frac12\rank \bbE\udot$. 

Note that any cosection $\sigma\ldot : Rp_*R\hom_{X \times W}(\cE,\cE)[1] \to \sO_X[-1]$ on $X$ induces a cosection on $X_L$ by the composition
$$\sigma^L\ldot : Rp_*R\hom_{X_L \times W}(\cE,\cE)_0[1]\lra Rp_*R\hom_{X_L \times W}(\cE,\cE)[1] \mapright{\sigma\ldot} \sO_{X_L}[-1]$$
and we have
$$({\sigma^L\ldot})\dual[E]=\sigma\ldot\dual([E])-\frac{1}{c_0}\id\circ \tr(\sigma\ldot\dual[E]) \in Ext^2(E,E)_0.$$
Hence the cosections $\sigma\ldot^\theta,\sigma\ldot^\delta,\sigma\ldot^\gamma$ in Examples \ref{151},  \ref{152}, and \ref{153} descend to cosections $\sigma\ldot^{\theta,L},\sigma\ldot^{\delta,L},\sigma\ldot^{\gamma,L}$ on $X_L$ and we have
\begin{align*}
({\sigma\ldot^{\theta,L}}){\dual}([E]) & = ({\sigma\ldot^\theta})\dual([E])-\frac{1}{c_0}\id(\theta\wedge c_2) = \theta\wedge At^2(E)-\frac{1}{c_0}\id(\theta\wedge c_2)\\
({\sigma\ldot^{\delta,L}})\dual([E]) & = ({\sigma\ldot^\delta})\dual([E])-\frac{1}{c_0}\id(\iota_\delta c_1)= \iota_\delta At(E)-\frac{1}{c_0}\id(\iota_\delta c_1)\\
({\sigma\ldot^{\gamma,L}})\dual([E]) & = ({\sigma\ldot^\gamma})\dual([E])-\frac{1}{c_0}\id(c_0\gamma)=0.
\end{align*}

Immediately, we have the following consequences.

\begin{coro}\label{158}
Let $W$ be a hyperk\"ahler 4-fold. If $c_1=c_2=0$ and $c_3\neq 0$, then the virtual cycle $[X_L]\virt$ vanishes and there is a reduced virtual cycle $[X_L]\virt_\redd \in A_{vd+1} (X_L)$.
\end{coro}

\begin{coro}\label{159}
Assume that $c_1=c_2=0$. Then any $(3,1)$-form $\delta \in H^1(W,\Omega^3_W)$ gives rise to a localized virtual cycle $[X]\virtloc \in A_{vd} (X({\sigma^{\delta,L}}))$.
\end{coro}

We omit the proofs of Corollary \ref{158} and Corollary \ref{159} since they are identical to those of Corollary \ref{196} and Corollary \ref{199}.

\begin{exam}\label{160}
Let $P_{\beta,n}$ be the moduli space of stable pairs on a Calabi-Yau 4-fold $W$ (cf. \cite{CMT2}). Note that $c=(1,0,0,-\beta,-n)$ and $\beta \neq 0$. Corollary \ref{158} and Corollary \ref{159} imply the following:
\begin{enumerate}
\item If $W$ is a hyperk{\"a}hler 4-fold, then $[P_{\beta,n}]\virt=0$ and there is a reduced virtual cycle $$[P_{\beta,n}]\virt_{\redd} \in A_{n+1}(P_{\beta,n}).$$
\item Any $(3,1)$-form $\delta$ on $W$ gives us a localized virtual cycle $$[P_{\beta,n}]\virtloc \in A_n(P_{\beta,n}({\sigma^\delta})).$$
\end{enumerate}
\end{exam}

\bigskip

\section{K-theoretic cosection localization and more}\label{S8}

In this section, we lift the results in the previous sections to the virtual structure sheaves in K-theory and more.  

\medskip

\subsection{Square root Euler class in K-theory}

In this subsection, we collect necessary facts on K-theory.

Let $Y$ be a scheme. Let $K^0(Y)$ (resp. $K_0(Y)$) be the Grothendieck group of vector bundles (resp. coherent sheaves) on $Y$, and $opK^0(Y)$ be the operational K-theory of Anderson-Payne \cite{AP}. 
Then there is a canonical map 
$$K^0(X) \lra opK^0(X).$$ 
Let $L$ be a line bundle. The {\em square root} of $L$ is
\begin{equation}
\sqrt{L}:= 1-\sum_{i \geq 1}a_i(1-L)^{\otimes i} \in K^0(X), \quad a_i=\frac{1}{i\cdot 2^{2i-1}}{\binom{2i-2}{i-1}} \in \QQ.    
\end{equation}
Then $L=\sqrt{L}^2$ and $\sqrt{L\otimes L'}=\sqrt{L}\cdot\sqrt{L'}$ for any line bundles $L,L'$.
The {\em Euler class} of a vector bundle $V$ is defined as
\begin{equation}
\fe(V)=\wedge_{-1}(V^\vee)=\sum_{i\geq 0}(-1)^i \wedge^i V^\vee \in K^0(X).
\end{equation}
Then $\fe(F)=\fe(F')\cdot\fe(F'')$ for any exact $0 \to F' \to F \to F'' \to 0$.

Let $F$ be an oriented $SO(2n)$-bundle over $Y$. Then we have an operational class 
\begin{equation}\label{170}
\sqrt{\fe}(F) \in opK^0(Y)    
\end{equation}
called the {\em square root Euler class} \cite[Definition 5.6]{OhTh} of $F$ satisfying the following property: For any morphism $f:Y' \to Y$ and a positive maximal isotropic subbundle $\Lambda$ of $f^*F$,
\begin{equation}
f^*\sqrt{\fe}(F)=\sqrt{\det(\Lambda)}\cdot\fe(\Lambda) \in opK^0(Y').    
\end{equation}
which was proved in \cite{OhTh} by using \cite[Appendix B. Theorem 3]{And}. 
For any $f:Y' \to Y$, we have $f^*\sqrt{\fe}(F)=\sqrt{\fe}(f^*F).$ If $K$ is an isotropic subbundle of $F$, then by \cite[(85)]{OhTh}, we have 
\begin{equation}\label{170.1}
\sqrt{\fe}(F)=\sqrt{\mathrm{det}(K)}\cdot\fe(K)\cdot\sqrt{\fe}(K^{\perp}/K).    
\end{equation}

\medskip



\subsection{Localized square root Euler class}\label{S8.2}

In this subsection, we lift the results in \S\ref{Sn4} and \S\ref{Scos} to K-theory. 

\begin{defi}\label{171}
Let $F$ be an $SO(2n)$-bundle over a scheme $Y$, and $s$ be an isotropic section of $F$. Using the notation of Definition \ref{n39}, we define the {\em localized square root Euler class} as  
\begin{equation}\label{171.1}
\sqrt{\fe}(F,s): K_0(Y) \to K_0(X)  ,   
\end{equation}
\begin{equation}\label{171.2}
\sqrt{\fe}(F,s)(\rho_*\alpha +\imath_*\beta) =\hat{\rho}_*\jmath^*(\sqrt{L}\cdot \sqrt{\fe}(\tF)\cdot \alpha)+\sqrt{\fe}(F)\beta
\end{equation}
which is an operational K-theory class in $opK^0(X \to Y)$ satisfying $$\imath_*\circ \sqrt{\fe}(F,s)=\sqrt{\fe}(F).$$ 
\end{defi}

Since the blowup sequence
$$K_0(D) \to K_0(\tY) \oplus K_0(X ) \to K_0(Y) \to 0$$
is also exact in K-theory \cite[Proposition 18.3.2]{Ful}, the proofs of Lemmas \ref{n43}, \ref{n44} and \ref{102} also work in K-theory after replacing the formula \eqref{n42} by \eqref{170.1}. 
Moreover, if $K$ is an isotropic subbundle of $F$ such that $s \cdot K=0$, then
\begin{equation}
\sqrt{\fe}(F,s)=\sqrt{\mathrm{det}(K)}\fe(K,s_2)\sqrt{\fe}(K^{\perp}/K,s_1)    
\end{equation}
as in Lemma \ref{103}, where $s_1 \in H^0(Y,K^{\perp}/K)$ and $s_2 \in H^0(s_1^{-1}(0),K)$ are the induced sections. As a corollary, we have 
\begin{equation}
\sqrt{\det(K)}\sqrt{\fe}((K^{\perp}/K)|_{C/K},\tau_1) = \sqrt{\fe}(F|_C,\tau) \circ p^* 
\end{equation}
under the assumptions of Corollary \ref{105}. The proofs are the same as those in \S\ref{Sn4}.
Furthermore the same arguments in \S\ref{SOT}  prove that \eqref{171.1} coincides with 
the K-theoretic Oh-Thomas class in \cite[\S5.2]{OhTh}.

Next we further localize $\sqrt{\fe}(F,s)$ by an additional section $t$ as in \S\ref{Scos}. 
\begin{defi}\label{172}
Let $F$ be an $SO(2n)$-bundle over a scheme $Y$, and $s,t$ be isotropic sections, independent away from a closed subscheme $Z\subset Y$ (cf. Definition \ref{n73}), such that $t^{-1}(0) \subseteq Z$ and $s\cdot t=0$. Using the notation of Definition \ref{n72}, the {\em localized square root Euler class} is defined as
\begin{equation}\label{172.1}
\sqrt{\fe}(F,s;t) : K_0(Y) \to K_0(X\cap Z), 
\end{equation}
\begin{equation}\label{172.2}
\sqrt{\fe}(F,s;t)(\rho_*\alpha+\imath_*\beta) = \rho''_* \sqrt{\fe}(\tF,\tit)  \jmath^*(\sqrt{L}\cdot\alpha) + \imath''_* \sqrt{\fe}(F,t)\beta.    
\end{equation}
\end{defi}

By the same proof, Theorem \ref{n74} holds for the K-theory class $\sqrt{\fe}(F,s;t)$. 
Under the assumptions of Lemma \ref{116}, for any $\xi \in K_0(Y)$, we have 
$$\sqrt{\fe}(F,s;t) \xi = \sqrt{\det(K)}\,\fe(K,s_2)\, \sqrt{\fe}(K^{\perp}/K,s_1,t_1) \xi \ \    \in\  K_0(X\cap Z).$$
Finally, under the assumptions of Corollary \ref{117}, for any $\xi\in K_0(C/K)$,
\begin{equation}\label{172.3}
\sqrt{\det(K)}\sqrt{\fe}(K^{\perp}/K|_{C/K},\tau_1;t_1) \xi = \sqrt{\fe}(F,\tau;t) p^* \xi \ \    \in\  K_0( Z).
\end{equation}
The proofs are identical to those in \S\ref{Scos} and we leave them to the reader. 

\subsection{Localized virtual structure sheaves and reduced virtual structure sheaves}\label{S8.4}

Let $X$ be a scheme with a symmetric obstruction theory $\phi:\bbE\udot \to \bbL_X$ perfect of amplitude $[-2,0]$ and an orientation. Under Assumption \ref{n86}, Oh-Thomas in \cite{OhTh} define the {\em (twisted) virtual structure sheaf} \cite{OhTh} as
\begin{equation}\label{n97}
[\sO_X\virt] = \sqrt{\det(B)\dual}\cdot \sqrt{\fe}(F|_C,\tau)[\sO_C]   \in K_0 (X)
\end{equation}
where $\bbE\udot\cong [B \to F \to B\dual]$ is a symmetric resolution, $C=\fC_X\times_{[F/B]}F$, and $\tau \in H^0(C,F|_C)$ is the tautological section.

Recall that for a cosection $\sigma\ldot :\bbE\ldot \to \sO_X[-1]$, we have a lift $\tsig:F \to \sO_X$ of $\sigma=h^1(\bsig):Ob_X\to \sO_X$ by Lemma \ref{13}.

\begin{theo}\label{173}
Under the assumptions of Theorem \ref{n88}, we have the {\em localized virtual structure sheaf}  defined as
\begin{equation}\label{173.1}
[\sO_{X,\loc}\virt] := \sqrt{\det(B)\dual}\cdot \sqrt{\fe}(F|_{Y},\tau;\tsig|_Y^\vee)[\sO_C]\ \ \in \ \ 
K_0( X({\sigma}))
\end{equation}
satisfying $\imath_*[\sO_{X,\loc}\virt]= [\sO_{X}\virt]$ where $Y=C_\redd$. 
\end{theo}

Replacing \eqref{117.1} by \eqref{172.3} in the proof of Lemma \ref{124}, we find that the localized virtual structure sheaf \eqref{173.1} is independent of the choices of a symmetric resolution of $\bbE\udot$ and the lifting $\tsig$.

In particular, if the cosection  $\sigma\ldot:\bbE\ldot \to \sO_X[-1]$ is isotropic and ${\sigma}=h^1(\bsig):Ob_X \to \sO_X$ is surjective, then $X(\sigma)=\emptyset$ and we have the vanishing $[\sO_X\virt]=0$. 

\begin{defi}\label{174}
When $\bsig$ is isotropic and $\sigma:Ob_X\to \sO_X$ is surjective, the {\em reduced virtual structure sheaf} is defined as
\begin{equation}\label{174.1}
[\sO_{X,\redd}\virt] := \sqrt{\det(B)\dual}\cdot \sqrt{\fe}(F_{\tsig}|_{Y},\tau)[\sO_C]\ \  \in\ \  K_0(X)
\end{equation}
where $F_{\tsig}$ is \eqref{n93} and $\tau \in H^0(Y,F_{\tsig}|_Y)$ is the tautological section.
\end{defi}

The reduced virtual structure sheaf \eqref{174.1} is independent of the choices of the symmetric resolution of $\bbE\udot$ and the lifting $\tsig$.

We also have the K-theoretic version of Theorem \ref{143} by the same proof. 
\begin{theo}\label{175}
If $\sigma\ldot^2 \in \CC^*$, then the virtual structure sheaf vanishes: 
\begin{equation}\label{175.1}
[\sO_X\virt]=0\ \  \in\ \  K_0(X).    
\end{equation}
\end{theo}

\medskip

\subsection{Virtual classes in general intersection theories} \label{S8.3}

In this subsection, we lift the virtual classes \eqref{n97}, \eqref{173.1} and \eqref{174.1} to more general intersection theories in the sense of \cite{KP}. 


Let $\cH_*$ be an intersection theory defined in \cite[Definition 2.2]{KP} for schemes. Then $\cH_*$ is an oriented Borel-Moore homology theory in the sense of Levine-Morel \cite[Definition 5.1.3]{LeMo}. Conversely, any oriented Borel-Moore homology theory which is detected by smooth schemes and has excision property is an intersection theory for schemes  (see \cite[Definition 1.1]{KP}).

There is a formal power series $g(u) \in \cH_*(\spec \CC)[[u]]$ such that for any line bundle $L$ over a scheme $X$, we have (cf. \cite{LeMo})
$$c_1(L\dual) = g(c_1(L)) c_1(L) : \cH_*( X) \to \cH_{*-1} (X)$$
such that $g(c_1(L\dual))=g(c_1(L))^{-1}$. Let $s_1,s_2,\cdots,s_n \in \ZZ[u_1,\cdots,u_n]$ be the elementary symmetric polynomials. 
Since the series $\prod_{1 \leq i \leq n} g(u_i)$ is symmetric, there is a formal power series $h(s_1,\cdots,s_n) \in \cH_*(\spec \bk)[[s_1,\cdots,s_n]]$ such that
$$h(s_1,\cdots,s_n) = (-1)^n \prod_{1 \leq i \leq n} g(u_i) \in \cH_*(\spec \bk)[[u_1,\cdots,u_n]].$$
Note that $h(0,\cdots,0)=1$ since $g(0)=-1$. Hence there exists a unique power series $\sqrt{h}(s_1,\cdots,s_n) \in \cH_*(\spec \bk)[[s_1,\cdots,s_n]]$ such that 
$$\sqrt{h}(s_1,\cdots,s_n)^2=h(s_1,\cdots,s_n)$$
and $\sqrt{h}(0,\cdots,0)=1$. For any vector bundle $V$ of rank $r$ over a scheme $Y$, define an operation
\begin{equation}\label{177}
\sqrt{h}(V) := \sqrt{h}(c_1(V),c_2(V),\cdots,c_n(V)) : \cH_*(X) \lra \cH_{*} (X).    
\end{equation}
By the splitting principle \cite[Remark 4.1.2]{LeMo}, we have
\begin{equation}
c_n(V\dual) = (-1)^n \sqrt{h}(V)^2 c_n(V): \cH_*(X) \lra \cH_{*-n}(X).    
\end{equation}
Therefore we have
\begin{equation}
\sqrt{h}(V\dual)c_n(V\dual) = (-1)^n \sqrt{h}(V)c_n(V) : \cH_* (X) \lra \cH_{*-n}(X)   
\end{equation}
because $\sqrt{h}(V\dual)=\sqrt{h}(V)^{-1}$. Also the splitting principle gives us
$$\sqrt{h}(V)=\sqrt{h}(V')\sqrt{h}(V'')$$
for any short exact sequence $0 \to V' \to V \to V'' \to 0$ of vector bundles.

We will need the following. 
\begin{assu}\label{176}
Let $\cH_*$ be an intersection theory for schemes satisfying the following two conditions:
\begin{enumerate}
\item (Kimura sequence) If $p:Y \to X$ is a proper surjective morphism of schemes, then we have a right exact sequence
\begin{equation}\label{176.1}
    \cH_* (Y\times_X Y) \xrightarrow{p_{1 *}- p_{2 *}} \cH_* (Y) \xrightarrow{p_*} \cH_* (X) \lra 0.
\end{equation}
\item (Fulton's conjecture) Let $F$ be an $SO(2n)$-bundle. If $\Lambda_1$ and $\Lambda_2$ are positive maximal isotropic subbundles of $F$, then we have the equality 
\begin{equation}\label{176.2}
    \sqrt{h}(\Lambda_1)c_n(\Lambda_1) = \sqrt{h}(\Lambda_2)c_n(\Lambda_2) :\cH_*(X) \lra \cH_{*-n}(X).
\end{equation}
\end{enumerate}
\end{assu}

\begin{exam}\label{178}
Let $A_*$ be the Chow homology in \cite{Ful} with $\QQ$-coefficients. Then $A$ is an intersection theory for schemes (cf.\cite[Example 2.4]{KP}). The right exact sequence \eqref{176.1} was proved in \cite[Theorem 1.8]{Kimu}. In this case, $\sqrt{h}(V)=1$ for any vector bundle $V$ of rank $r$. The formula \eqref{176.2} was proved in \cite[Theorem 1 (c)]{EdGr}. Hence the Chow homology $A_*$ satisfies Assumption \ref{176}.
\end{exam}

\begin{exam}\label{179}
Let $K_0[\beta,\beta^{-1}]$ be the algebraic K-theory with $\QQ$-coefficients. For any scheme $X$, we assign  $K_0(X)[\beta,\beta^{-1}]=\bigoplus_{d \in \ZZ} K_0(X)\cdot \beta^d$. Then  $K_0[\beta,\beta^{-1}]$ is an intersection theory for schemes (cf. \cite[Example 2.5]{KP}). The Kimura sequence \eqref{176.1} follows from the Riemann-Roch theorem \cite[Theorem 18.2]{Ful}. For a line bundle $L$ over a scheme $X$, we have
$$c_1(L\dual)\alpha=(1-[L])\cdot\alpha\cdot\beta=-[L](1-[L\dual])\cdot\alpha\cdot\beta=-[L]\cdot c_1(L)\alpha$$
for any $\alpha \in K_0(X)$. Therefore we have $g(c_1(L))\alpha=-[L]\cdot\alpha$ and  
$$\sqrt{h}(V)=[\sqrt{\det V}].$$
Fulton's conjecture \eqref{176.2} holds for algebraic K-theory by \cite[Appendix B. Theorem 3]{And}. Therefore the algebraic K-theory $K_0[\beta,\beta^{-1}]$ satisfies Assumption \ref{176}.
\end{exam}

\begin{exam}\label{180}
Let $\Omega_*$ be the  algebraic cobordism theory with $\QQ$-coefficents (cf. \cite{LeMo}). Then $\Omega_*$ is an intersection theory for schemes (cf. \cite[Theorem 2.6]{KP}). By \cite[Theorem 4.1.28]{LeMo} and \cite[Theorem 4.5.1]{LeMo}, we have the Kimura sequence \eqref{176.1}.
\end{exam}

\begin{ques}\label{181}
Does Fulton's conjecture 
\eqref{176.2} hold for algebraic cobordism $\Omega_*$?
\end{ques}

From now on, we assume that $\cH_*$ is an intersection theory for schemes satisfying Assumption \ref{176}.

\begin{defi}\label{182}
Let $F$ be an $SO(2n)$-bundle over a scheme $Y$. Let $p:Q \to Y$ be the full flag variety of isotropic subbundles \eqref{55}. Let $\Lambda$ be the universal positive maximal isotropic subbundle of $F|_Q$. We define the {\em square root Euler class}
\begin{equation}\label{182.1}
    \sqrt{e}(F):\cH_*(Y) \lra \cH_{*-n}(Y)
\end{equation}
 of $F$ to be the unique map which fits into the commutative diagram
$$\xymatrix{
\cH_* (Q\times_YQ) \ar[r]^(.6){p_{1 *}- p_{2 *}}\ar[d]_{\sqrt{h}(p_1^*\Lambda)e(p_1^*\Lambda)} & \cH_* (Q) \ar[r]^{p_*} \ar[d]^{\sqrt{h}(\Lambda)e(\Lambda)} & \cH_* (Y) \ar[r]\ar@{.>}[d]^{\sqrt{e}(E)} & 0\\
\cH_{*-n} (Q\times_YQ) \ar[r]^(.6){p_{1 *}- p_{2 *}} & \cH_{*-n} (Q) \ar[r]^{p_*}  & \cH_{*-n} (Y) \ar[r] & 0.
}$$
where $e(\Lambda)=c_n(\Lambda)$. Here the rows are exact by \eqref{176.1} and the left square commutes because by  \eqref{176.2} we have 
$$\sqrt{h}(p_1^*\Lambda) e(p_1^*\Lambda)=\sqrt{h}(p_2^*\Lambda)e(p_2^*\Lambda).$$
\end{defi}

It is straightforward that $\sqrt{e}(F)$ commutes with projective pushforwards and refined lci pullbacks. For an isotropic subbundle $K$ of $F$, we have
\begin{equation}
\sqrt{e}(F) = \sqrt{h}(K) e(K)\sqrt{e}(K^\perp/K).    
\end{equation}


\begin{lemm}\cite[Lemma 7.9]{Vishik} \label{185}
Let $p:\tY \to Y$ be a projective morphism, which is an isomorphism away from a closed subscheme $X \subseteq Y$. Let $D=X\times_Y\tY$. Then we have a right exact sequence
\begin{equation}\label{185.1}
\cH_*(D) \lra \cH_*(X) \oplus \cH_*(\tY) \lra \cH_*(Y) \to 0.
\end{equation}
\end{lemm}

Now we can define a localized square root Euler class in $\cH_*$.

\begin{defi}\label{184}
Let $F$ be an $SO(2n)$-bundle over a scheme $Y$ and $s\in H^0(F)$ be an isotropic section. 
Using the notation of Definition \ref{n39}, we define the {\em localized square root Euler class}  as 
\begin{equation}\label{184.1}
\sqrt{e}(F,s):\cH_*(Y) \lra \cH_{*-n}(X)    
\end{equation}
\begin{equation}\label{184.2}
\sqrt{e}(F,s)(\rho_*\alpha +\imath_*\beta) = \hat{\rho}_* \sqrt{h}(L)\jmath^*\sqrt{e}(\tF)\alpha +\sqrt{e}(F)\beta.    
\end{equation}
\end{defi}

The blowup sequence \eqref{185.1} proves that the results in \S\ref{Sn4} can be generalized to an intersection theory $\cH_*$. Indeed, \eqref{184.2} is well defined and $\imath_*\circ\sqrt{e}(F,s)=\sqrt{e}(F)$. The square root Euler class $\sqrt{e}(F,s)$ commutes with projective pushforwards and refined lci pullbacks. Under the assumptions of Lemma \ref{103}, we have
\begin{equation}
\sqrt{e}(F,s)=\sqrt{h}(K)e(K,s_2)\sqrt{e}(K^\perp/K,s_1).    
\end{equation}

\begin{theo}\label{188}
Under the assumptions of Theorem \ref{n74}, the {\em localized square root Euler class}
\begin{equation}\label{188.1}
    \sqrt{e}(F,s;t) : \cH_*(X) \lra \cH_{*-n} (X\cap Z)
\end{equation}
\begin{equation}\label{188.2}
    \sqrt{e}(F,s;t) (\rho_*\alpha+\imath_*\beta) =\rho''_*\sqrt{h}(L)\jmath^*\sqrt{e}(\tF,\tit)\alpha +\imath''_*\sqrt{e}(F,t)\beta
\end{equation}
is well defined and $\imath'_*\circ \sqrt{e}(F,s;t)=\sqrt{e}(F,s)$ holds.
Moreover $\sqe(F,s;t)$ commutes with projective pushforwards and refined lci pullbacks. 
\end{theo}
Under the assumptions of of Lemma \ref{116}, for $ \xi\in \cH_*(Y)$, we have
$$\sqrt{e}(F,s;t)\xi = \sqrt{h}(K)e(K,s_2)\sqrt{e}(K^\perp/K,s_1,t_1)\xi\ \ \in\ \ \cH_* (X\cap Z).$$

\begin{defi}\label{186}
Under Assumption \ref{n86}, the \emph{virtual fundamental class} of $X$ is defined as
\begin{equation}
[X]\virt:=\sqrt{h}(B)^{-1}\sqrt{e}(F|_Y,\tau)q^*\mathrm{sp}[\spec \bk] \in \cH_*(X).    
\end{equation}
where $q:C \to \fC_X$ is the projection, $\tau$ is the tautological section and $\mathrm{sp}:\cH_*(\spec \bk) \to \cH_*(\fC_X)$ be the specialization map in \cite[Definition 4.5]{KP}. 
\end{defi}
Note that $q^*\mathrm{sp}[\spec \bk]$ plays the role of the fundamental class of $C$. 

\begin{theo} 
Under Assumption \ref{n86}, 
if $\sigma\ldot$ is isotropic, the {\em localized virtual fundamental class} 
\begin{equation}
[X]\virtloc := \sqrt{h}(B)^{-1}\sqrt{e}(F|_{Y},\tau;\tsig|_Y\dual)q^*\mathrm{sp}[\spec \bk]\  \in\  \cH_* (X({\sigma}))
\end{equation}
is well defined and we have $$\imath_*[X]\virtloc=[X]\virt.$$
In particular, if $\sigma\ldot$ is isotropic and ${\sigma}$ is surjective, then $[X]\virt=0$ and we have the {\em reduced virtual fundamental class} defined as 
\begin{equation}
[X]\virt_\redd:=\sqrt{h}(B)^{-1}\sqrt{e}(F_{\tsig}|_{Y},\tau_{\tsig})q^*\mathrm{sp}[\spec \bk]\  \in\  \cH_{*} (X),    
\end{equation}
where $F_{\tsig}$ is the reduction of $F$ by $\langle\tsig\dual\rangle$ and $\tau$ is the tautological section.  Moreover, if $\sigma\ldot^2 \in \CC^*$, then $[X]\virt=0$.
\end{theo}
The proofs are parallel to those in \S\ref{S5n}.

\bigskip

\bibliographystyle{amsplain}

\end{document}